\documentclass[leqno,11pt]{amsart}

\usepackage{amsmath,amstext,amssymb,amsopn,amsthm,mathrsfs}

\title[C-Z operators related to Jacobi expansions]
    {Calder\'on-Zygmund operators related to Jacobi expansions}
    
\author[A{.} Nowak]{Adam Nowak}
\author[P{.} Sj\"ogren]{Peter Sj\"ogren}

\address{Adam Nowak, \newline
			Institute of Mathematics,
      Polish Academy of Sciences, \newline
      \'Sniadeckich 8,
      00--956 Warszawa, Poland \newline
			\indent and \newline
			Institute of Mathematics and Computer Science, \newline
      Wroc\l{}aw University of Technology,       \newline
      Wyb{.} Wyspia\'nskiego 27,
      50--370 Wroc\l{}aw, Poland      
      }
\email{anowak@pwr.wroc.pl}

\address{
\noindent Peter Sj\"ogren \newline
    Mathematical Sciences, 
    University of Gothenburg \newline
    Mathematical Sciences,
    Chalmers University of Technology \newline 
    SE-412 96 G\"oteborg
    Sweden}
\email{\noindent peters@chalmers.se}

\allowdisplaybreaks
\theoremstyle{plain}
\newtheorem{thm}{Theorem}[section]

\newtheorem{lem}[thm]{Lemma}
\newtheorem{prop}[thm]{Proposition}
\newtheorem{cor}[thm]{Corollary}

\theoremstyle{definition}

\theoremstyle{remark}
\newtheorem*{rem*}{Remark}
\newtheorem{rem}[thm]{Remark}

\theoremstyle{plain}

\DeclareMathOperator{\domain}{Dom}

\DeclareMathOperator{\support}{supp}

\DeclareMathOperator*{\essup}{ess\,sup}

\def\R{\mathbb R}

\def\d{d}
\def\P{\mathcal P}
\def\H{\mathcal H}

\def\m{\mu} 						
\def\ab{\alpha,\beta}
\def\J{\mathcal J} 			
\def\q{\mathfrak q}

\begin{document}

\begin{abstract}
We study several fundamental operators in harmonic analysis related to Jacobi expansions,
including Riesz transforms, imaginary powers of the Jacobi operator, the Jacobi-Poisson semigroup maximal
operator and Littlewood-Paley-Stein square functions. We show that these are (vector-valued)
Calder\'on-Zygmund operators in the sense of the associated space of homogeneous type,
and hence their mapping properties follow from the general theory.
Our proofs rely on an explicit formula for the Jacobi-Poisson kernel, which we derive from
a product formula for Jacobi polynomials.
\end{abstract}

\maketitle

\footnotetext{
		
		The first-named author was supported in part by MNiSW Grant N N201 417839.
}

\section{Introduction} \label{sec:intro}

The fundamental paper \cite{MuS} of B{.} Muckenhoupt and E{.}M{.} Stein initiated in 1965 an important
development in harmonic analysis known as \emph{harmonic analysis of orthogonal expansions}. The principal 
part of \cite{MuS} is devoted to classical ultraspherical expansions. Recently this setting was
reinvestigated by means of more modern techniques in \cite{Bur1,Bur2}.
In the present paper we treat the general Jacobi setting, which is a natural
generalization of the ultraspherical one. In fact, the suggestion of further research
in this direction appears explicitly in \cite[p.\,22]{MuS}.
The point of view in \cite{MuS} is shaped by the classical Fourier analysis
in the torus and has deep roots in the interplay between Fourier series, analytic functions
and harmonic functions. Here as in \cite{Bur1,Bur2}, we adopt the spectral point of view,
which seems more natural and appropriate from a time perspective and was systematically applied
later in the seminal monograph of Stein \cite{topics}. This manifests itself in slight differences
between objects arising naturally according to these two points of view. 
Some aspects of harmonic analysis related to the Jacobi setting in the spirit of \cite{MuS} were
studied earlier by Li \cite{Li}, and recently by Stempak \cite{Stem}. 
However, our approach, governed by the general Calder\'on-Zygmund theory,
is different and in fact much wider, and it seems more modern.

Given parameters $\alpha,\beta>-1$, we consider the Jacobi differential operator
$$
\J^{\ab} = - \frac{d^2}{d\theta^2} - \frac{\alpha-\beta+(\alpha+\beta+1)\cos\theta}{\sin \theta}
	\frac{d}{d\theta} + \Big( \frac{\alpha+\beta+1}{2}\Big)^2
$$
on the interval $(0,\pi)$ equipped with the (doubling) measure
$$
d\m_{\ab}(\theta) = \Big( \sin\frac{\theta}2 \Big)^{2\alpha+1} 
	\Big( \cos\frac{\theta}2\Big)^{2\beta+1} d\theta.
$$
This operator is formally symmetric and positive in $L^2(d\m_{\ab})$,
and its spectral decomposition is discrete and is given by the classical Jacobi polynomials,
see Section \ref{sec:prel} for details. Moreover, $\J^{\ab}$ admits the decomposition
$$
\J^{\ab} = \delta^*\delta + \Big( \frac{\alpha+\beta+1}{2}\Big)^2,
$$
where $\delta = {d}\slash{d\theta}$ and $\delta^*$ is the formal adjoint of $\delta$ in $L^2(d\m_{\ab})$.
For the special choice of $\alpha=\beta=\lambda-1\slash 2$, the situation reduces to
the ultraspherical setting of type $\lambda$ considered in \cite{MuS,Bur1,Bur2}.

The central objects of our study are the following linear or sublinear operators related to $\J^{\ab}$
(for strict definitions see Section \ref{sec:prel}).
\begin{itemize}
\item[(i)] Imaginary powers of the Jacobi operator
$$
I_{\gamma}^{\ab}\colon f \mapsto \big(\J^{\ab}\big)^{-i\gamma}f, \qquad \gamma \in \R,\quad \gamma\neq 0.
$$
\item[(ii)] Riesz-Jacobi transforms of arbitrary order $N$
$$
R_N^{\ab}\colon f \mapsto \delta^{N} \big(\J^{\ab}\big)^{-N\slash 2}f, \qquad N=1,2,\ldots.
$$
\item[(iii)] The Jacobi-Poisson semigroup maximal operator
$$
\H_{*}^{\ab}\colon f\mapsto \big\| \exp\big(-t\sqrt{\J^{\ab}}\big)f\big\|_{L^{\infty}(dt)}.
$$
\item[(iv)] The vertical and horizontal square functions based on the Jacobi-Poisson semigroup
\begin{align*}
g_{V}^{\ab}\colon & f\mapsto \big\| \partial_t \exp\big(-t\sqrt{\J^{\ab}}\big)f \big\|_{L^2(tdt)},\\
g_{H}^{\ab}\colon & f\mapsto \big\| \delta \exp\big(-t\sqrt{\J^{\ab}}\big)f \big\|_{L^2(tdt)}.
\end{align*}
\item[(v)] Mixed square functions of arbitrary orders $M,N$ based on the Jacobi-Poisson semigroup
$$
g_{M,N}^{\ab}\colon f \mapsto\big\| \partial_t^{M} \delta^{N} 
\exp\big(-t\sqrt{\J^{\ab}}\big)f \big\|_{L^2(t^{2M+2N-1}dt)},
$$
where $M,N=0,1,2,\ldots$ and $M+N>1$. 
\end{itemize}
Our main result, Theorem \ref{thm:main} below, says that under the slight restriction 
$\alpha,\beta\ge -1\slash 2$ these operators are scalar-valued or can be viewed as vector-valued
Calder\'on-Zygmund  operators in the sense of the space of homogeneous type
$((0,\pi),d\m_{\ab},|\cdot|)$, where $|\cdot|$ stands for the ordinary distance. 
Consequences of this, including mapping properties in weighted $L^p$ spaces, are then delivered
by the general theory.
The present results constitute a natural extension of those mentioned above in the ultraspherical
setting \cite{MuS,Bur1,Bur2} and complement those on Riesz transforms and conjugacy in the Jacobi
setting \cite{Li,Stem}. Further comments can be found at the end of Section \ref{sec:prel}.

The main difficulty related to the Calder\'on-Zygmund approach is to obtain suitable kernel
estimates. Inspired by earlier ideas used in certain Laguerre settings \cite{Sa,NS},
we present a transparent technique based on a convenient symmetric double-integral
representation of the Jacobi-Poisson kernel emerging from the product formula 
for Jacobi polynomials due to Dijksma and Koornwinder \cite{DK}.
This method is of independent interest and is in fact applicable to a larger class of operators than
(i)-(v), including multipliers of Laplace transform type in the sense of Stein 
(see \cite[p.\,58, p.\,121]{topics}) and Lusin's square functions.
The well-known closed formula for the Jacobi-Poisson kernel in terms of Appel's hypergeometric
function, see Section~\ref{sec:prel}, does not seem to be useful in this context.
According to our knowledge, so far no reasonable representation is available for either the Jacobi
heat kernel or for the multi-dimensional Jacobi-Poisson kernel.
This is the main reason for limiting our investigations to objects expressible via the
one-dimensional Jacobi-Poisson semigroup.

The paper is organized as follows. Section \ref{sec:prel} contains the setup, strict definitions
of the operators (i)-(v), statements of the main results and accompanying comments and remarks.
In Section \ref{sec:L2}, the operators (i)-(v) are proved to be $L^2$-bounded and associated, in the
Calder\'on-Zygmund theory sense, with the relevant kernels. Finally, Section \ref{sec:ker}
is devoted to the proofs of all the necessary kernel estimates. This is the largest and most technical
part of the work.

Throughout the paper we use a standard notation with essentially all symbols referring to the
space of homogeneous type $((0,\pi),d\m_{\ab},|\cdot|)$. Since the distance in this space is the
Euclidean one, the ball denoted $B(r,\theta)$ is simply the interval $(\theta-r,\theta+r)\cap (0,\pi)$.
By $\langle f,g\rangle_{d\m_{\ab}}$ we mean
$\int_{(0,\pi)} f(\theta)\overline{g(\theta)}\, d\m_{\ab}(\theta)$ whenever the integral makes sense.
Further, $L^p(wd\m_{\ab})$ stands for the weighted $L^p$ space, $w$ being a nonnegative
weight on $(0,\pi)$. Given $1 \le p < \infty$, $p'$ is its adjoint exponent, $1\slash p + 1\slash p' =1$.
For $1\le p < \infty$, we denote by $A^{\ab}_p = A^{\ab}_p((0,\pi),d\m_{\ab})$ the Muckenhoupt class 
of $A_p$ weights related to the measure $\m_{\ab}$.  
More precisely, $A^{\ab}_p$ is the class of all nonnegative functions $w$ such that
\begin{equation*}
\sup_{I \in \mathcal{I}} \bigg[ \frac{1}{\m_{\ab}(I)} \int_I w(\theta) \,
     d\m_{\ab}(\theta) \bigg]  \bigg[ \frac{1}{\m_{\ab}(I)} \int_I w(\theta)^{-p'\slash p}
     \, d\m_{\ab}(\theta) \bigg]^{p\slash p'} < \infty
\end{equation*}
when $1<p<\infty$, or
\begin{equation*}
\sup_{I \in \mathcal{I}} \frac{1}{\m_{\ab}(I)} \int_I w(\theta) \,
     d\m_{\ab}(\theta) \; \essup_{\theta \in I} \frac{1}{w(\theta)} < \infty
\end{equation*}
if $p=1$; here $\mathcal{I}$ is the class of all subintervals of $(0,\pi)$. 
Clearly, this implies that $w \in L^1(d\m_{\ab})$. It is easy to check that a double-power weight 
$w(\theta) = (\sin\frac{\theta}2)^r (\cos\frac{\theta}2)^s$ belongs to $A^{\ab}_p$, $1<p<\infty$, 
if and only if $-(2\alpha+2)<r<(2\alpha+2)(p-1)$ and $-(2\beta+2)<s<(2\beta+2)(p-1)$,
and $w \in A^{\ab}_1$ if and only if $-(2\alpha+2)< r \le 0$ and $-(2\beta+2)< s \le 0$.

While writing estimates, we will frequently use the notation $X \lesssim Y$
to indicate that $X \le CY$ with a positive constant $C$ independent of significant quantities.
We shall write $X \simeq Y$ when simultaneously $X \lesssim Y$ and $Y \lesssim X$.

\section{Preliminaries and statement of main results} \label{sec:prel}

Given $\alpha, \beta > -1$, the standard Jacobi polynomials
of type $\ab$ are defined on the interval $(-1,1)$ by the Rodrigues formula
$$
P_n^{\ab} (x) = \frac{(-1)^n}{2^n n!} (1-x)^{-\alpha}(1+x)^{-\beta}
    \frac{{\d}^n}{{\d}x^n} \big( (1-x)^{\alpha+n}(1+x)^{\beta+n} \big),
    \qquad n=0,1,2,\ldots.
$$
Note that each $P_n^{\ab}$ is a polynomial of degree $n$.
It is natural and convenient to apply the trigonometric parametrization 
$x=\cos\theta$, $\theta\in (0,\pi)$, and consider the normalized trigonometric polynomials
$$
\P_n^{\ab}(\theta) = c_n^{\ab} P_n^{\ab}(\cos\theta),
$$
with the normalizing factor
$$
c_n^{\ab}=\|P_n^{\ab}(\cos\theta)\|^{-1}_{L^2(d\m_{\ab}(\theta))} =
	\bigg( \frac{(2n+\alpha+\beta+1)\Gamma(n+\alpha+\beta+1)\Gamma(n+1)}
		{\Gamma(n+\alpha+1)\Gamma(n+\beta+1)} \bigg)^{1\slash 2},
$$
where for $n=0$ and $\alpha+\beta=-1$
the product $(2n+\alpha+\beta+1)\Gamma(n+\alpha+\beta+1)$ must be replaced by
$\Gamma(\alpha+\beta+2)$. It is well known that the system $\{\P_n^{\ab}:n\ge 0\}$
is orthonormal and complete in $L^2((0,\pi),d\m_{\ab})$. 
Moreover, each $\P_n^{\ab}$ is an eigenfunction of the Jacobi operator, 
$$
\J^{\ab} \P_n^{\ab} = \lambda_n^{\ab} \P_n^{\ab}, 
	\qquad \lambda_n^{\ab}=\Big( n+ \frac{\alpha+\beta+1}{2}\Big)^2.
$$
Thus $\J^{\ab}$, considered initially on $C_c^{\infty}(0,\pi)$, has a natural self-adjoint
extension in $L^2(d\m_{\ab})$, still denoted by the same symbol $\J^{\ab}$ and given by
\begin{equation} \label{sres}
\J^{\ab} f = \sum_{n=0}^{\infty} \lambda_n^{\ab} \langle f, \P_n^{\ab} \rangle_{d\m_{\ab}} \P_n^{\ab}
\end{equation}
on the domain $\domain\J^{\ab}$ consisting of all functions $f\in L^2(d\m_{\ab})$ 
for which the defining series converges in $L^2(d\m_{\ab})$. 
Then the spectral decomposition of $\J^{\ab}$ is given by \eqref{sres}.
To see that this produces an extension from $C_c^{\infty}(0,\pi)$, observe that
$\lambda_n^{\ab}\langle f, \P_n^{\ab}\rangle_{d\m_{\ab}}
= \langle \J^{\ab}f, \P_n^{\ab}\rangle_{d\m_{\ab}}$ for any $f\in C_c^{\infty}(0,\pi)$.

The semigroup generated by the square root of $\J^{\ab}$ is called the Jacobi-Poisson semigroup
and will be denoted by $\{\H_t^{\ab}\}$. We have for $f \in L^2(d\m_{\ab})$ and $t>0$
\begin{equation} \label{Hser}
\H_t^{\ab}f = \exp\Big(-t\sqrt{\J^{\ab}}\Big)f = \sum_{n=0}^{\infty} e^{-t|n+\frac{\alpha+\beta+1}{2}|}
	\langle f, \P_n^{\ab} \rangle_{d\m_{\ab}} \P_n^{\ab},
\end{equation}
the convergence being in $L^2(d\m_{\ab})$. 
In fact, the last series converges pointwise for any $f \in L^p(wd\m_{\ab})$, $1\le p < \infty$,
$w \in A_p^{\ab}$, and defines a smooth function of $(t,\theta) \in (0,\infty)\times (0,\pi)$.
To give a brief justification of this fact, we note that the normalized Jacobi polynomials
satisfy the estimate (see \cite[(7.32.2)]{Sz})
\begin{equation} \label{estjac}
|\P_n^{\ab}(\theta)| \lesssim (n+1)^{\alpha+\beta+5\slash 2}, \qquad \theta \in (0,\pi), \quad n \ge 0.
\end{equation}
Using H\"older's inequality, one proves that the Fourier-Jacobi coefficients of any 
$f\in L^p(wd\m_{\ab})$, $w\in A_p^{\ab}$, $1\le p < \infty$, grow at most polynomially, in the sense that
\begin{equation} \label{growthFJ}
\big|\langle f, \P_n^{\ab}\rangle_{d\m_{\ab}}\big| \lesssim \|f\|_{L^p(wd\m_{\ab})} 
	(n+1)^{\alpha + \beta + 5\slash 2}, \qquad n \ge 0.
\end{equation}
Therefore, the series in \eqref{Hser} converges absolutely and uniformly because of
the exponentially decreasing factor. Moreover, term by term differentiation of this series together
with the differentiation rule (cf. \cite[(4.21.7)]{Sz})
\begin{equation} \label{jacdiff}
\frac{d}{d\theta} \P_n^{\ab}(\theta) = -\frac{1}{2}\sqrt{n(n+\alpha+\beta+1)}\, \sin\theta \;
	\P_{n-1}^{\alpha+1,\beta+1}(\theta), \qquad n \ge 0,
\end{equation}
shows that it defines a smooth function of $(t,\theta)\in (0,\infty)\times (0,\pi)$.
In \eqref{jacdiff}, and elsewhere, we use the convention that $\P^{\ab}_{k}\equiv 0$ if $k<0$.
Thus the series \eqref{Hser} can be regarded as a definition of $\H_t^{\ab}$ on the weighted spaces
$L^p(wd\m_{\ab})$, $w\in A_p^{\ab}$, $1\le p < \infty$.

The integral representation of $\{\H_t^{\ab}\}$, valid on the weighted $L^p$ spaces mentioned above
(see \cite{StTo} or \cite{No} for the relevant arguments), is
$$
\H_t^{\ab}f(\theta) = \int_0^{\pi} H_t^{\ab}(\theta,\varphi)f(\varphi)\, d\m_{\ab}(\varphi),
	\qquad \theta \in (0,\pi), \quad t>0,
$$
with the Jacobi-Poisson kernel
\begin{equation} \label{PJser}
H_t^{\ab}(\theta,\varphi) = \sum_{n=0}^{\infty} e^{-t|n+\frac{\alpha+\beta+1}{2}|}
	\P_n^{\ab}(\theta)\P_n^{\ab}(\varphi).
\end{equation}
The last series converges absolutely for all $\theta,\varphi \in (0,\pi)$ and $t>0$,
defining a smooth function of $(t,\theta,\varphi)\in (0,\infty)\times(0,\pi)^2$;
this follows from \eqref{estjac}, \eqref{jacdiff} and term by term differentiation. 
On the other hand, the series in \eqref{PJser} is highly oscillating. 
Since the behavior of the kernel is essentially hidden behind the oscillations,
to analyze objects involving $H_t^{\ab}(\theta,\varphi)$ we will need a more convenient representation.
It is well known that $H_t^{\ab}(\theta,\varphi)$ can be expressed by means of Appel's hypergeometric
function of two variables $F_4$. For $\alpha,\beta>-1$ such that $\alpha+\beta\ge -1$
\begin{align*}
H_t^{\ab}(\theta,\varphi) & = \frac{1}{2^{+\alpha+\beta+1}\m_{\ab}((0,\pi))} 
	\frac{\sinh\frac{t}{2}}{(\cosh\frac{t}{2})^{\alpha+\beta+2}} \\ & \quad \times
	  F_4\Bigg( \frac{\alpha+\beta+2}{2},\frac{\alpha+\beta+3}{2}; \alpha+1, \beta+1; 
	\bigg(\frac{\sin\frac{\theta}2\sin\frac{\varphi}2}{\cosh\frac{t}2}\bigg)^2,
	\bigg(\frac{\cos\frac{\theta}2\cos\frac{\varphi}2}{\cosh\frac{t}2}\bigg)^2\Bigg).
\end{align*}
This formula is due to Watson; it can also be obtained from a result of Bailey,
see \cite[p.\,385--387]{AAR}. From this expression, positivity and continuity with respect to 
the parameters $\alpha,\beta$ of the Jacobi-Poisson kernel can easily be seen.
However, for our purposes we need a more suitable representation, 
which will be derived in Section \ref{sec:ker}.

We now give precise definitions on $L^2(d\m_{\ab})$ of our main objects of interest.
For $f\in L^2(d\m_{\ab})$ we define
\begin{itemize}
\item[(i)] imaginary powers of the Jacobi operator
$$
I^{\ab}_{\gamma}f = \sum_{n=0}^{\infty} \Big| n + \frac{\alpha+\beta+1}{2}\Big|^{-2\gamma i}
	\langle f, \P_n^{\ab} \rangle_{d\m_{\ab}} \P_n^{\ab}, 
$$
where $\alpha+\beta \neq -1$, $\gamma \in \R$, $\gamma \neq 0$;
\item[(ii)] Riesz-Jacobi transforms of order $N$
$$
R_N^{\ab}f = \sum_{n=0}^{\infty} \Big| n + \frac{\alpha+\beta+1}{2}\Big|^{-N}
	\langle f, \P_n^{\ab} \rangle_{d\m_{\ab}} \delta^N \P_n^{\ab}, 
$$
where $\alpha+\beta \neq -1$ and $N=1,2,\ldots$;
\item[(iii)] the Jacobi-Poisson semigroup maximal operator
$$
\H^{\ab}_* f(\theta) = \big\|\H_t^{\ab}f(\theta)\big\|_{L^{\infty}(dt)}, \qquad \theta \in (0,\pi);
$$
\item[(iv)] the vertical and horizontal square functions based on the Jacobi-Poisson semigroup
\begin{align*}
g_V^{\ab}(f)(\theta) & = \big\|\partial_t \H_t^{\ab}f(\theta)\big\|_{L^2(tdt)}, \qquad \theta \in (0,\pi),\\
g_H^{\ab}(f)(\theta) & = \big\|\delta \H_t^{\ab}f(\theta)\big\|_{L^2(tdt)}, \qquad \theta \in (0,\pi);
\end{align*}
\item[(v)] mixed square functions of arbitrary orders based on the Jacobi-Poisson semigroup
$$
g_{M,N}^{\ab}(f)(\theta) = \big\|\partial_t^M \delta^N \H_t^{\ab}f(\theta)\big\|_{L^2(t^{2M+2N-1}dt)},
$$
where $M,N=0,1,2,\ldots$ and $M+N>0$.
\end{itemize}
Notice that (v) includes (iv) because $g_{V}^{\ab} = g_{1,0}^{\ab}$ and $g_H^{\ab}=g_{0,1}^{\ab}$.
Here and in the statements of the results we distinguish $g_V^{\ab}$ and $g_H^{\ab}$ since these
are the most common $g$-functions.
As will be explained in Section \ref{sec:L2}, $I^{\ab}_{\gamma}$ and $R^{\ab}_N$ are indeed
well defined on $L^2(d\m_{\ab})$ by the above formulas, since the series converge in $L^2(d\m_{\ab})$ and
the operators are bounded on $L^2(d\m_{\ab})$.
As for the remaining operators, their definitions are understood pointwise and 
are valid for general $f \in L^p(wd\m_{\ab})$, $w \in A_p^{\ab}$, $1\le p < \infty$,
since $\H_t^{\ab}f(\theta)$ is a smooth function of $(t,\theta)\in (0,\infty)\times (0,\pi)$.

We remark that in the so-called critical case when $\alpha+\beta=-1$ (and in particular in the 
fundamental case $\alpha=\beta=-1\slash 2$), $I_{\gamma}^{\ab}$ and $R_N^{\ab}$ cannot be defined by 
the above spectral formulas since then $0$ is an eigenvalue of $\J^{\ab}$. To deal with this obstacle, 
one usually considers these operators on the orthogonal complement of the eigenspace corresponding to 
the eigenvalue $0$. So letting $\Pi_0$ be the orthogonal projection onto $\{\P_0^{\ab}\}^{\perp}$, 
for $f \in L^2(d\m_{\ab})$ we can consider in the critical case
\begin{align*}
I_{\gamma}^{\ab} \Pi_0  f & = \sum_{n=1}^{\infty} \Big| n + \frac{\alpha+\beta+1}{2}\Big|^{-2\gamma i}
	\langle f, \P_n^{\ab} \rangle_{d\m_{\ab}} \P_n^{\ab}, \qquad \gamma \in \R, \quad \gamma \neq 0,\\
R_N^{\ab} \Pi_0 f & = \sum_{n=1}^{\infty} \Big| n + \frac{\alpha+\beta+1}{2}\Big|^{-N}
	\langle f, \P_n^{\ab} \rangle_{d\m_{\ab}} \delta^N \P_n^{\ab}, \qquad N=1,2,\ldots.
\end{align*}
These definitions are indeed correct, see Section \ref{sec:L2} below.
Moreover, since $\delta \P_0^{\ab}$ vanishes, the case of the Riesz transforms can actually be
covered by the definition in (ii) above. Thus in further considerations we will not distinguish 
the critical case and always denote the Riesz operators by $R_N^{\ab}$.

The operators $\H_*^{\ab}$, $g_V^{\ab}$, $g_H^{\ab}$, $g_{M,N}^{\ab}$ are not linear.
They are, however, associated with vector-valued linear operators taking values in some Banach space
$\mathbb{B}$. Indeed, it is convenient to identify each of them with a linear operator which maps
a scalar-valued function of $\theta \in (0,\pi)$ to a $\mathbb{B}$-valued function of $\theta$.
The corresponding nonlinear operator defined above is then obtained by taking the $\mathbb{B}$ norm
at each point $\theta$, or rather at a.a. $\theta$. Clearly, $\mathbb{B}$ will be $L^2(tdt)$ in the
cases of $g_V^{\ab}$ and $g_H^{\ab}$, and $L^2(t^{2M+2N-1}dt)$ in the case of $g_{M,N}^{\ab}$.
For $\H_{*}^{\ab}$ we shall, for technical reasons, choose $\mathbb{B}$ not as $L^{\infty}(dt)$ but
as the closed and separable subspace $\mathbb{X}\subset L^{\infty}(dt)$ consisting of all continuous
functions $f$ in $(0,\infty)$ which have finite limits as $t\to 0^+$ and as $t \to \infty$.
In all the four cases, we shall say that the operator is \emph{associated} with the corresponding
Banach space $\mathbb{B}$. Similarly, the linear operators $I_{\gamma}^{\ab}$ and $R_N^{\ab}$ will
be said to be associated with the Banach space $\mathbb{B}=\mathbb{C}$.

To obtain the boundedness results for our operators, we shall see that they are vector-valued
Calder\'on-Zygmund operators, in the sense that we now define. As always, this definition goes via the
kernel. So let $\mathbb{B}$ be a Banach space and let $K(\theta,\varphi)$ be a kernel defined on
$(0,\pi)\times (0,\pi) \backslash \{(\theta,\varphi):\theta=\varphi\}$ and taking values in $\mathbb{B}$.
We say that $K(\theta,\varphi)$ is a standard kernel in the sense of the space of homogeneous type
$((0,\pi),d\m_{\ab},|\cdot|)$ if it satisfies the growth estimate
\begin{equation} \label{gr}
\|K(\theta,\varphi)\|_{\mathbb{B}} \lesssim \frac{1}{\m_{\ab}(B(\theta,|\varphi-\theta|))}
\end{equation}
and the smoothness estimates
\begin{align}
\| K(\theta,\varphi) - K(\theta',\varphi)\|_{\mathbb{B}} 
	& \lesssim \frac{|\theta-\theta'|}{|\theta-\varphi|}\;
	\frac{1}{\m_{\ab}(B(\theta,|\varphi-\theta|))}, \qquad |\theta-\varphi| > 2|\theta-\theta'|,
	\label{sm1} \\
\| K(\theta,\varphi) - K(\theta,\varphi')\|_{\mathbb{B}} & 
	\lesssim \frac{|\varphi-\varphi'|}{|\theta-\varphi|}\;
	\frac{1}{\m_{\ab}(B(\theta,|\varphi-\theta|))}, \qquad |\theta-\varphi| > 2|\varphi-\varphi'|; \label{sm2}
\end{align}
here $B(\theta,r)$ denotes the ball (interval) centered at $\theta$ and of radius $r$.
When $K(\theta,\varphi)$ is scalar-valued, i.e. $\mathbb{B}=\mathbb{C}$, the difference conditions
\eqref{sm1} and \eqref{sm2} can be replaced by the more convenient gradient condition
\begin{equation} \label{sm}
|\partial_{\theta} K(\theta,\varphi)| + |\partial_{\varphi} K(\theta,\varphi)| \lesssim
\frac{1}{|\theta-\varphi| \m_{\ab}(B(\theta,|\varphi-\theta|))}.
\end{equation}
Notice that in these formulas, the ball $B(\theta,|\varphi-\theta|)$ can be replaced by
$B(\varphi,|\varphi-\theta|)$, in view of the doubling property of $\m_{\ab}$.

A linear operator $T$ assigning to each $f \in L^2(d\m_{\ab})$ a measurable $\mathbb{B}$-valued
function $Tf$ on $(0,\pi)$ is said to be a (vector-valued) Calder\'on-Zygmund operator in the sense of 
the space $((0,\pi),d\m_{\ab},|\cdot|)$ associated with $\mathbb{B}$ if
\begin{itemize}
\item[(a)] $T$ is bounded from $L^2(d\m_{\ab})$ to $L^2_{\mathbb{B}}(d\m_{\ab})$, and
\item[(b)] there exists a standard $\mathbb{B}$-valued kernel $K(\theta,\varphi)$ such that
\begin{equation} \label{id_CZ}
Tf(\theta) = \int_{(0,\pi)} K(\theta,\varphi) f(\varphi) \, d\m_{\ab}(\varphi), \qquad \textrm{a.e.}
	\;\; \theta \notin \support f,
\end{equation}
for every $f\in L^2((0,\pi),d\m_{\ab})$ with compact support in $(0,\pi)$.
\end{itemize}
When $(b)$ holds, we write $T \sim K(\theta,\varphi)$ and say that $T$ is associated with $K$.
Here integration of $\mathbb{B}$-valued functions is understood in Bochner's sense, and 
$L^2_{\mathbb{B}}(d\m_{\ab})$ is the Bochner-Lebesgue space of all $\mathbb{B}$-valued $d\m_{\ab}$-square
integrable functions on $(0,\pi)$. It is well known that a large part of the classical theory
of Calder\'on-Zygmund operators remains valid, with appropriate adjustments, when the underlying
space is of homogeneous type and the associated kernels are vector-valued, see for instance the comments
in \cite[p.\,649]{NS} and references given there.

The main result of the paper reads as follows.
\begin{thm} \label{thm:main}
Assume that $\alpha, \beta \ge -1\slash 2$. The operators $I_{\gamma}^{\ab}$, $\alpha+\beta > -1$, 
$\gamma \neq 0$, and $R_N^{\ab}$, $N=1,2,\ldots$, are Calder\'on-Zygmund operators in the sense of 
the space of homogeneous type $((0,\pi),d\m_{\ab},|\cdot|)$. 
Moreover, each of the operators $\H_{*}^{\ab}$, $g_V^{\ab}$, $g_H^{\ab}$ and $g_{M,N}^{\ab}$,
$M,N=0,1,2,\ldots$, $M+N>1$, viewed as a vector-valued operator, is a Calder\'on-Zygmund operator 
in the sense of the space $((0,\pi),d\m_{\ab},|\cdot|)$ associated with $\mathbb{B}$, and here 
$\mathbb{B}$ is $\mathbb{X}$, $L^2(tdt)$, $L^2(tdt)$ or $L^2(t^{2M+2N-1}dt)$, respectively.
\end{thm}

The proof of Theorem \ref{thm:main} splits naturally into the following three results.

\begin{prop} \label{prop:L2}
Let $\alpha,\beta \ge -1\slash 2$. The operators $I_{\gamma}^{\ab}$, $\alpha+\beta > -1$, 
$\gamma \neq 0$, $R_N^{\ab}$, $N=1,2,\ldots$, $\H_{*}^{\ab}$, $g_V^{\ab}$, $g_H^{\ab}$, and
$g_{M,N}^{\ab}$, $M,N=0,1,2,\ldots$, $M+N>1$, are bounded on $L^2(d\m_{\ab})$.
In particular, each of the operators $\H_{*}^{\ab}$, $g_V^{\ab}$, $g_H^{\ab}$,
$g_{M,N}^{\ab}$, $M,N=0,1,2,\ldots$, $M+N>1$, viewed as a vector-valued operator,
is bounded from $L^2(d\m_{\ab})$ to $L^2_{\mathbb{B}}(d\m_{\ab})$, where $\mathbb{B}$ 
is as in Theorem~\ref{thm:main}. Moreover, the operators $g_{V}^{\ab}=g_{1,0}^{\ab}$ and more 
generally $g_{M,0}^{\ab}$ are essentially isometries on $L^2$ in the sense that
$$
\|g_{M,0}^{\ab}(f)\|_{L^2(d\m_{\ab})} = c \|f\|_{L^2(d\m_{\ab})}
$$
with $c=c(M)$; however, in the case when $\alpha=\beta=-1\slash 2$, 
one must replace $f$ by $\Pi_0 f$ in the right-hand side here.
\end{prop}

For $\alpha,\beta \ge -1\slash 2$ define the kernels
\begin{align}
K^{\ab}_{\gamma}(\theta,\varphi) & = \frac{1}{\Gamma(2i\gamma)} \int_0^{\infty}
	H_t^{\ab}(\theta,\varphi) t^{2i\gamma-1}dt, 
		\qquad \gamma \in \R,\quad \gamma \neq 0, \quad \alpha+\beta > -1, \label{imag_ker} \\
R_N^{\ab}(\theta,\varphi) & = \frac{1}{\Gamma(N)}\int_0^{\infty} 
	\partial_{\theta}^N H_t^{\ab}(\theta,\varphi) t^{N-1}\,dt, \qquad N\ge 1. \nonumber
\end{align}

\begin{prop} \label{prop:assoc}
Let $\alpha,\beta \ge -1\slash 2$. The operators
$I_{\gamma}^{\ab}$, $\alpha+\beta > -1$, $\gamma \neq 0$, and $R_N^{\ab}$, 
$N=1,2,\ldots$, are associated with the following kernels:
$$
I_{\gamma}^{\ab} \sim K^{\ab}_{\gamma}(\theta,\varphi), \qquad 
R_N^{\ab} \sim R_N^{\ab}(\theta,\varphi).
$$
Further, the operators $\H_{*}^{\ab}$, $g_{V}^{\ab}$, $g_{H}^{\ab}$, $g_{M,N}^{\ab}$,
$M,N=0,1,2,\ldots$, $M+N>1$, viewed as vector-valued operators,
are associated with the following $\mathbb{B}$-valued kernels:
$$
\begin{array}{lll}
& \H_{*}^{\ab}   \sim \{H_t^{\ab}(\theta,\varphi)\}_{t>0}, \qquad 
& g_{V}^{\ab}   \sim \{\partial_t H_t^{\ab}(\theta,\varphi)\}_{t>0}, \\
& g_{H}^{\ab}   \sim \{\partial_{\theta} H_t^{\ab}(\theta,\varphi)\}_{t>0}, \qquad
& g_{M,N}^{\ab}  \sim \{\partial_t^{M}\partial_{\theta}^N H_t^{\ab}(\theta,\varphi)\}_{t>0}.
\end{array}
$$
Here $\mathbb{B}$ is as in Theorem \ref{thm:main}.
\end{prop}

\begin{thm} \label{thm:stand}
Assume that $\alpha,\beta \ge -1\slash 2$. The scalar-valued kernels $K^{\ab}_{\gamma}(\theta,\varphi)$, 
$\gamma \in \R$, $\gamma \neq 0$, $\alpha+\beta > -1$, and $R_N^{\ab}(\theta,\varphi)$, $N=1,2,\ldots$,
satisfy the standard estimates \eqref{gr}, with $\mathbb{B}=\mathbb{C}$, and \eqref{sm}.
Further, the vector-valued kernels appearing in Proposition \ref{prop:assoc} satisfy
the standard estimates \eqref{gr}, \eqref{sm1} and \eqref{sm2}, with $\mathbb{B}$ as before.
\end{thm}

The proofs of Propositions \ref{prop:L2} and \ref{prop:assoc} are given in Section \ref{sec:L2}.
The proof of Theorem \ref{thm:stand} is the most technical part of the paper and is located in
Section \ref{sec:ker}. The restriction $\alpha,\beta \ge -1\slash 2$ in the results
is imposed by the method we use to prove the standard estimates in Theorem \ref{thm:stand}, 
and more precisely by a similar restriction in the fundamental formula of Dijksma and Koornwinder 
needed to derive suitable expressions for the kernels; see Section \ref{sec:ker} for details.

An important consequence of Theorem \ref{thm:main} is the following.

\begin{cor} \label{cor:main}
Let $\alpha,\beta \ge -1\slash 2$. Then each of the operators 
$I_{\gamma}^{\ab}$, $\alpha+\beta > -1$, $\gamma \neq 0$, $R_N^{\ab}$, 
$N=1,2,\ldots$, $\H_{*}^{\ab}$, $g_V^{\ab}$, $g_H^{\ab}$, and
$g_{M,N}^{\ab}$, $M,N=0,1,2,\ldots$, $M+N>1$, is bounded
on $L^p(wd\m_{\ab})$, $w\in A_p^{\ab}$, $1<p<\infty$, and from $L^1(wd\m_{\ab})$ 
to weak $L^1(wd\m_{\ab})$, $w \in A_1^{\ab}$.
\end{cor}

Further consequences of Theorem \ref{thm:main} can be derived from the general theory of
Calder\'on-Zygmund operators, see for instance \cite[Section 1]{BFS} and references given there.
We leave this to interested readers.

\begin{proof}[Proof of Corollary \ref{cor:main}]
The assertions for the scalar-valued operators $I_{\gamma}^{\ab}$ and $R_N^{\ab}$ 
follow from the standard Calder\'on-Zygmund theory for spaces of homogeneous type. 

The case of the maximal operator $\H_{*}^{\ab}$ is analogous to the Laguerre heat-diffusion maximal
operator considered in \cite{NS}, see the proof of \cite[Theorem 2.1]{NS}. 
The relevant fact that each $\H_t^{\ab}$, $t>0$, is bounded on $L^p(wd\m_{\ab})$, $w\in A_p^{\ab}$,
$1\le p < \infty$, can be easily justified by means of \eqref{estjac} and \eqref{growthFJ}.

Finally, the assertions for the square functions $g_V^{\ab}$, $g_{H}^{\ab}$, 
$g_{M,N}^{\ab}$, are proved by means of \eqref{estjac} and
\eqref{growthFJ}, by the arguments used in the Hermite and Laguerre function
settings; see the proofs of \cite[Theorem 2.2]{StTo2} and \cite[Corollary 2.5]{Szarek}.
\end{proof}

We finish this section with various remarks.

\begin{rem}
It is not appropriate to replace $\delta$ by its adjoint 
$\delta^*=-\delta-(\alpha+\frac{1}{2})\cot\frac{\theta}2+(\beta+\frac{1}{2})\tan\frac{\theta}2$
in the definitions of the Riesz-Jacobi transforms and the square functions involving the
horizontal component.
Focus for instance on $g_{H}^{\ab}$. Let 
$\widetilde{g}_H^{\ab}(f)(\theta)=\|\delta^* \H_t^{\ab}f(\theta)\|_{L^2(tdt)}$ be the $g$-function
arising by replacing $\delta$ with $\delta^*$ in the definition of $g_H^{\ab}$.
A direct computation reveals that, for $\alpha+\beta> -1\slash 2$,
$$
\widetilde{g}_H^{\ab}(\P_0^{\ab})(\theta) = |\alpha+\beta+1|^{-1} \bigg|
	\Big(\alpha+\frac{1}{2}\Big) \cot\frac{\theta}2 - \Big(\beta+\frac{1}2\Big) \tan\frac{\theta}2\bigg|
		\P_0^{\ab}(\theta).
$$
Since $\P_0^{\ab}\in L^p(d\m_{\ab})$ for all $p \ge 1$ and 
$\widetilde{g}_H^{\ab}(\P_0^{\ab})\notin L^p(d\m_{\ab})$ when $p\ge \min(2\alpha+2,2\beta+2)$, we see that 
$\widetilde{g}_H^{\ab}$ is not bounded on all the spaces $L^p(d\m_{\ab})$, $1<p<\infty$.
In particular, it follows that $\widetilde{g}_H^{\ab}$ cannot be a Calder\'on-Zygmund operator.
\end{rem}

\begin{rem}
Lower $L^p$ estimates, $1<p<\infty$, for the Jacobi vertical square functions can be deduced 
from Corollary \ref{cor:main} and a standard duality argument, see \cite[p.\,66--67]{MuS}.
This can be generalized to the weighted setting with $A_p^{\ab}$ weights admitted, see
\cite[Remark 2.6]{Szarek}.
\end{rem}

\begin{rem}
In connection with the critical case $\alpha+\beta=-1$ excluded in Theorem \ref{thm:main}, we note that for
$\alpha=\beta = -1\slash 2$ the operator $I_{\gamma}^{\ab}\Pi_0$ can be treated like $I_{\gamma}^{\ab}$, 
$\alpha+\beta \neq -1$, and proved to be a Calder\'on-Zygmund operator associated with the kernel
$$
\widetilde{K}_{\gamma}^{-1\slash 2,-1\slash 2}(\theta,\varphi) 
	=  \frac{-1}{\Gamma(1+2i\gamma)} \int_0^{\infty}
	\partial_t H_t^{-1\slash 2,-1\slash 2}(\theta,\varphi)
	 t^{2i\gamma}dt, 	\qquad \gamma \in \R,\quad \gamma \neq 0.
$$
Details are left to interested readers.
\end{rem}

Finally, we relate our results to those in the earlier papers mentioned in the introduction and
concerning the Jacobi setting.

In the article \cite{Bur2}, the authors consider the ultraspherical setting with the
type parameter $\lambda > 0$, which coincides with our Jacobi setting with
$\alpha=\beta=\lambda-1\slash 2>-1\slash 2$. In the main result, they prove that the corresponding
Riesz transforms of arbitrary order are Calder\'on-Zygmund operators in the sense of the
associated space of homogeneous type. Moreover, they show that certain square functions,
see \cite[(1.4),(1.5)]{Bur2}, can be viewed as vector-valued Calder\'on-Zygmund operators.
Our present results extend those from \cite{Bur2} in several directions.
First of all, we consider the Jacobi setting with arbitrary $\alpha,\beta \ge -1\slash 2$;
in particular, the case $\alpha=\beta=-1\slash 2$ is included. Secondly, we deal with an essentially
wider variety of operators, including imaginary powers, the Poisson semigroup maximal operator and 
mixed square functions. For $\alpha=\beta$ the Riesz-Jacobi
transforms $R_N^{\ab}$ coincide with the Riesz operators from \cite{Bur2}, the vertical $g$-function
$g_V^{\ab}$ coincides with the $g$-function in \cite[(1.4)]{Bur2}, and the horizontal $g$-function
$g_H^{\ab}$ dominates that in \cite[(1.5)]{Bur2}. 
The technique for proving standard estimates developed in this paper is different
even in the ultraspherical setting and seems more transparent. 
In addition, we give a shorter proof of the $L^2$-boundedness of the Riesz-Jacobi transforms.

The first-order Riesz transform related to ultraspherical expansions of type $\lambda>0$
was investigated earlier, by different methods, in \cite{Bur1}, and, among other results,
the $L^p$-boundedness, $1<p<\infty$, and weak type $(1,1)$ were obtained in the unweighted context.

Considering the fundamental paper \cite{MuS}, we already mentioned that our definitions 
of operators, as well as those in \cite{Bur1,Bur2}, are ``spectral'' and differ from those in \cite{MuS}.
However, they are related and this allows one to move certain results in both directions.
For instance the ultraspherical Riesz transform of order $1$ and thus also our Riesz-Jacobi transform
are related to the conjugate function mapping from \cite{MuS} by a well-behaved multiplier operator,
see \cite[Section 6]{Bur1}. Further, the $g$-function studied in \cite{MuS} can be treated, at least 
partially, by means of the Calder\'on-Zygmund theory and the technique of kernel estimates presented in
Section \ref{sec:ker}, see \cite[Section 4.3]{Bur2}. 
The $L^p$ and weak-type $(1,1)$ boundedness of the Poisson integral maximal operator proved in
\cite{MuS} can also be obtained from our result about the maximal operator $\H^{\ab}_{*}$.

Some results of \cite{MuS} were generalized to the Jacobi setting in \cite{Li}, in particular
the $L^p$ mapping properties of the conjugate function mapping. The proof in \cite{Li}
is based on deep estimates of the transplantation kernel for Jacobi orthonormalized polynomials
obtained by Muckenhoupt \cite{Mu}. In fact the result on the conjugate function mapping in \cite{Li}
(and thus also in \cite{MuS}) is a direct consequence of Muckenhoupt's transplantation theorem 
\cite{Mu}, see \cite[Section 5]{Stem}.

We mention that conjugacy problems in other Jacobi settings were investigated earlier by Stempak
\cite{StemT} in the ultraspherical case and by the authors \cite{NoSj}.
Recently, Betancor et al. \cite{BFRT} complemented the results of \cite{Bur2} by deriving 
principal-value integral representations for the ultraspherical Riesz transforms of higher orders.

The imaginary powers of the Jacobi operator can be viewed as spectral multipliers related to $\J^{\ab}$. 
Consequently, some special cases of our main result on $I_{\gamma}^{\ab}$ are
covered by multiplier theorems existing in the literature. In particular,
$I_{\gamma}^{\ab}$ is a multiplier of Laplace transform type in the
sense of Stein \cite{topics} and hence, in the ultraspherical case, its unweighted
$L^p$-boundedness and weak type $(1,1)$ follow from the result of Mart\'{\i}nez,
see \cite[Theorem 1.1]{Ma}. For an account of other multiplier theorems in the ultraspherical
and Jacobi settings, we refer to \cite[Section 1]{Ma}. Finally, we observe that
with only slightly more effort, our methods are sufficient for proving, 
via the Calder\'on-Zygmund theory, a weighted multiplier theorem for Jacobi expansions in the spirit 
of Stein's general multiplier theorem for contraction semigroups \cite[Corollary 3, p.\,121]{topics}.
We leave the details to interested readers.

\section{$L^2$-boundedness and kernel associations} \label{sec:L2}

In this section we show that the operators that we are dealing with are indeed well defined
and bounded on $L^2(d\m_{\ab})$.
Then we identify the kernels which these operators are associated with
in the Calder\'on-Zygmund theory sense.

\begin{proof}[Proof of Proposition \ref{prop:L2}; the cases of $I_{\gamma}^{\ab}$ and $\H_{*}^{\ab}$]
The Plancherel theorem shows that the imaginary powers $I_{\gamma}^{\ab}$ are well-defined
isometries on $L^2(d\m_{\ab})$, except when $\alpha=\beta=-1/2$. In the latter case, the
$I_{\gamma}^{\ab}\Pi_0$ are isometries on $\{\P_0^{\ab}\}^{\perp}$ and contractions on $L^2(d\m_{\ab})$.

Next, we observe that the $L^2$-boundedness of the maximal operator $\H_*^{\ab}$ is a consequence of the
analogous property for the Jacobi-Poisson maximal operator 
$S_*^{\ab}$ in the standard Jacobi polynomial setting, see \cite[p.\,346]{NoSj}. 
This is because $\H_*^{\ab}$ can be controlled pointwise by that maximal operator.
Indeed, letting $\{T_t^{\ab}\}$ be the one-dimensional Jacobi semigroup in the setting of \cite{NoSj}
and $\{\widetilde{T}_t^{\ab}\}$ be the semigroup generated by $\J^{\ab}$, we have
$$
\widetilde{T}_t^{\ab} (f\circ \cos) (\theta) = e^{-t(\frac{\alpha+\beta+1}{2})^2} T_t^{\ab}f(\cos\theta),
	\qquad \theta \in (0,\pi),
$$
for suitable functions $f$ on $(-1,1)$. Then the subordination principle implies 
$|\H_*^{\ab}(f\circ \cos)(\theta)| \le S_*^{\ab}|f|(\cos\theta)$, and the conclusion follows.
\end{proof}

The treatment of the Riesz transforms is less straightforward. It is not even clear whether the defining
series converges in $L^2(d\m_{\ab})$, and this is because $\{\delta^N \P_n^{\ab}:n=0,1,2,\ldots\}$ does not
form an orthogonal system unless $N=1$.
To overcome this obstacle, we will decompose $\delta^N \P_n^{\ab}$ into a suitable sum involving
other orthogonal systems, see \cite[Section 4]{NoSt} for a general background.
For this purpose, the following formula is crucial (cf. \cite[(7.27)]{NoSt})
\begin{align} \nonumber
A \cos\theta \, \P_{n-1}^{\alpha+1,\beta+1}(\theta) & 
	= B \Big( \sin\frac{\theta}2 \cos\frac{\theta}2\Big)^2
	\P_{n-2}^{\alpha+2,\beta+2}(\theta) 
	+ C \Big(\sin\frac{\theta}2\Big)^2 \P_{n-1}^{\alpha+2,\beta}(\theta) \\
	& \quad + D \Big(\cos\frac{\theta}2\Big)^2 \P_{n-1}^{\alpha,\beta+2}(\theta) 
	+ E \,\P_n^{\ab}(\theta), \qquad n \ge 1,\label{ffor}
\end{align}
with the coefficients
\begin{align*}
A & = (\alpha+1)(\beta+1)(\alpha+\beta+2n),\\ 
B & = \sqrt{(n-1)(n+\alpha+\beta+2)} \big( (n-1)(\alpha+\beta+2) + (\alpha+1)^2+(\beta+1)^2 \big),\\
C & = (\beta-\alpha) \sqrt{(n+\beta)(n+\alpha+1)} (n+\alpha),\\
D & = (\beta-\alpha) \sqrt{(n+\alpha)(n+\beta+1)} (n+\beta),\\
E & = \sqrt{n(n+\alpha+\beta+1)} \big( (n-1)(\alpha+\beta+2) + 2(\alpha+1)(\beta+1) \big).
\end{align*}
Notice that $A \simeq n$, and that $B=\mathcal{O}(n^2)$, $C=\mathcal{O}(n^2)$,
$D=\mathcal{O}(n^2)$ and $E=\mathcal{O}(n^2)$ as $n \to \infty$. 
\begin{lem} \label{decomp}
Let $N \ge 1$. Then we have the decomposition
\begin{equation} \label{dec}
\delta^N \P_n^{\ab}(\theta) = \sum_{0\le \nu,\eta,p \le 2N} \mathcal{O}(n^N)
	\Big(\sin\frac{\theta}2\Big)^{\nu} \Big(\cos\frac{\theta}2\Big)^{\eta}
	\P_{n-p}^{\alpha+\nu,\beta+\eta}(\theta).
\end{equation}
\end{lem}

Here and in the sequel, we write expressions like $\mathcal{O}(n^N)$ for factors which are independent
of $\theta$ and which are bounded in modulus by $C n^N$ with $C$ independent of $n$,
and also independent of $M$ in connection with the operators $g_{M,N}^{\ab}$.

\begin{proof}[Proof of Lemma \ref{decomp}]
First we claim that
\begin{equation} \label{c1}
\delta^N \P_n^{\ab}(\theta) = \sum_{\substack{1\le r \le N\\0 \le m \le r}}
	\mathcal{O}(n^{N-r+m}) (\sin\theta)^m (\cos\theta)^{r-m}
	\P_{n-r}^{\alpha+r,\beta+r}(\theta).
\end{equation}
To verify this we will use Fa\`a di Bruno's formula for the $N$th derivative of the composition
of two functions (see \cite{Jo} for the related references and interesting historical remarks),
\begin{equation} \label{Faa}
\partial_{\theta}^N(g\circ f)(\theta) = \sum \frac{N!}{k_1! \cdots k_N!} \;\partial^{k_1+\ldots+k_N}
	g\circ f(\theta) \bigg( \frac{\partial^1 f(\theta)}{1!}\bigg)^{k_1}\cdots
	\bigg( \frac{\partial^N f(\theta)}{N!}\bigg)^{k_N},
\end{equation}
where the summation runs over all $k_1,\ldots,k_N \ge 0$ such that $k_1+2k_2+\ldots+N k_N = N$.
Choosing $f(\theta) = \cos\theta$ and $g(x)=g_n^{\ab}(x)=\P_n^{\ab}(\arccos x)$, and using the fact that
$$
\partial_x g(x) = -\frac{1}2 \sqrt{n(n+\alpha+\beta+1)}\P_{n-1}^{\alpha+1,\beta+1}(\arccos x)
	= \mathcal{O}(n) g_{n-1}^{\alpha+1,\beta+1}(x),
$$
which follows from the differentiation rule \eqref{jacdiff}, we see that
$$
\delta^N \P_n^{\ab}(\theta) = \sum_{k_1+2k_2+\ldots+Nk_N=N} \mathcal{O}(n^{|k|})
	(\sin\theta)^{\sum_{\textrm{odd}\, i \le N}k_i} (\cos\theta)^{\sum_{\textrm{even}\, i \le N}k_i}
	\P_{n-|k|}^{\alpha+|k|,\beta+|k|}(\theta).
$$
Since the powers of $\sin\theta$ and $\cos\theta$ sum to $|k|$ and the constraint 
$k_1+2k_2+\ldots+Nk_N=N$ implies $|k|\le N-\sum_{\textrm{even}\, i \le N}k_i$,
the claim follows if we let $|k|=r$ and $\sum_{\textrm{odd}\, i \le N}k_i=m$.

Next, we claim that for $r-m \ge 1$
\begin{align} \nonumber 
&(\cos\theta)^{r-m} \P_{n-r}^{\alpha+r,\beta+r}(\theta) \\ &= \sum_{s_1,s_2\in\{0,1\}^{r-m}}
	\mathcal{O}(n^{r-m}) \Big( \sin\frac{\theta}2\Big)^{2|s_1|}\Big(\cos\frac{\theta}2\Big)^{2|s_2|}
	\P_{n-m-|s_1|-|s_2|}^{\alpha+m+2|s_1|,\beta+m+2|s_2|}(\theta), \label{c2}
\end{align}
where $|s_1|,|s_2|$ denote the lengths of the multi-indices $s_1,s_2$. Indeed, by \eqref{ffor} we get
$$
\cos\theta \, \P_{n-r}^{\alpha+r,\beta+r}(\theta) = \sum_{s_1,s_2\in\{0,1\}}
	\mathcal{O}(n) \Big( \sin\frac{\theta}2\Big)^{2 s_1}\Big(\cos\frac{\theta}2\Big)^{2 s_2}
	\P_{n-r+1-s_1-s_2}^{\alpha+r-1+2s_1,\beta+r-1+2s_2}(\theta)
$$
and iterating this we arrive precisely at \eqref{c2}.

A combination of \eqref{c1} and \eqref{c2}, and the fact that 
$\sin \theta = 2\sin \frac{\theta}2 \cos \frac{\theta}2$, reveal that
$$
\delta^N \P_n^{\ab}(\theta) = \sum \mathcal{O}(n^N) 
	\Big( \sin\frac{\theta}2\Big)^{m+2|s_1|}\Big(\cos\frac{\theta}2\Big)^{m+2|s_2|}
	\P_{n-m-|s_1|-|s_2|}^{\alpha+m+2|s_1|,\beta+m+2|s_2|}(\theta),
$$
where the summation runs over $0 \le m \le N$ and $s_1,s_2 \in \{0,1\}^{N-m}$, possibly with some 
terms vanishing. This implies the assertion of the lemma.
\end{proof}

We note that \eqref{dec} could be improved, since some of the terms vanish.
\begin{proof}[Proof of Proposition \ref{prop:L2}; the case of $R_N^{\ab}$]
Using Lemma \ref{decomp} and taking into account the fact that each of the systems
\begin{equation} \label{adorth}
\Big\{ \Big(\sin\frac{\theta}2\Big)^{\nu} \Big(\cos\frac{\theta}2\Big)^{\eta}
	\P_n^{\alpha+\nu,\beta+\eta}(\theta) : n=0,1,2,\ldots \Big\}, \qquad \eta, \nu \ge 0,
\end{equation}
is orthonormal in $L^2(d\m_{\ab})$, we infer (see \cite[Proposition 3]{NoSt}) that the operators 
$R_N^{\ab}$ are well defined and bounded on $L^2(d\m_{\ab})$.
\end{proof}

We proceed to square functions. 
\begin{proof}[Proof of Proposition \ref{prop:L2}; the cases of $g_V^{\ab}$, $g_H^{\ab}$ and $g_{M,N}^{\ab}$]
It is enough to verify the boundedness in $L^2(d\m_{\ab})$ of
$$
g_{M,N}^{\ab}(f)(\theta) = \big\|\partial_t^M \delta^N \H_t^{\ab}f(\theta)\big\|_{L^2(t^{2M+2N-1}dt)}
$$
with any $M,N=0,1,2,\ldots$ such that $M+N>0$. 

By differentiating the series defining $\H_t^{\ab}f$ in \eqref{Hser},  
we get for $f \in L^2(d\m_{\ab})$
$$
\partial_t^M \delta^N \H_t^{\ab}f = \sum_{n=0}^{\infty} (-1)^M \Big|n+\frac{\alpha+\beta+1}2\Big|^M
	e^{-t|n+\frac{\alpha+\beta+1}2|} \langle f,\P_n^{\ab} \rangle_{d\m_{\ab}} \delta^N \P_n^{\ab},
$$
which in view of Lemma \ref{decomp} gives
\begin{align*}
&\partial_t^M \delta^N \H_t^{\ab}f(\theta) \\ & = \sum_{0\le \nu,\eta,p\le 2N}
\sum_{n=0}^{\infty} \mathcal{O}(n^{M+N})
	e^{-t|n+\frac{\alpha+\beta+1}2|} \langle f,\P_n^{\ab} \rangle_{d\m_{\ab}} 
	\Big(\sin\frac{\theta}2\Big)^{\nu} \Big(\cos\frac{\theta}2\Big)^{\eta}
	\P_{n-p}^{\alpha+\nu,\beta+\eta}(\theta).
\end{align*}
In the exceptional case $\alpha+\beta+1=0$, there is no term with $n=0$ in these and the next few sums.
Now the orthonormality of the systems \eqref{adorth} leads to
\begin{align*}
\big\| g_{M,N}^{\ab}(f)\big\|^2_{L^2(d\m_{\ab})} & = \int_0^{\pi}\int_0^{\infty}
	\big| \partial_t^M\delta^N \H_t^{\ab}f(\theta)\big|^2 t^{2M+2N-1} dtd\m_{\ab}(\theta) \\
& \lesssim \int_0^{\infty} \sum_{n=0}^{\infty} n^{2M+2N} 
	e^{-2t|n+\frac{\alpha+\beta+1}2|} t^{2M+2N-1}|\langle f,\P_n^{\ab} \rangle_{d\m_{\ab}}|^2 dt \\
& = \Gamma(2M+2N)\sum_{n=0}^{\infty}  \frac{n^{2M+2N}}{(2n+\alpha+\beta+1)^{2M+2N}}
	|\langle f,\P_n^{\ab} \rangle_{d\m_{\ab}}|^2 \\
& \lesssim \|f\|^2_{L^2(d\m_{\ab})}.
\end{align*}

Finally, to prove the claimed isometry property of $g_{M,0}^{\ab}$, notice that 
$$
\partial_t^M \H_t^{\ab}f = \sum_{n=0}^{\infty} (-1)^M \Big|n+\frac{\alpha+\beta+1}2\Big|^M
	e^{-t|n+\frac{\alpha+\beta+1}2|} \langle f,\P_n^{\ab} \rangle_{d\m_{\ab}} \P_n^{\ab}.
$$
Then Parseval's theorem shows that, for $f \in L^2(d\m_{\ab})$,
\begin{align*}
\big\| g_{M,0}^{\ab}(f)\big\|^2_{L^2(d\m_{\ab})} & =
\sum_{n=0}^{\infty} \Big|n+\frac{\alpha+\beta+1}2\Big|^{2M} 
	|\langle f,\P_n^{\ab} \rangle_{d\m_{\ab}}|^2
	\int_0^{\infty} e^{-2t|n+\frac{\alpha+\beta+1}2|} t^{2M-1}dt  \\
& = 2^{-2M}\Gamma(2M) \|f\|^2_{L^2(d\m_{\ab})};
\end{align*}
when $\alpha+\beta+1=0$ the last occurrence of $f$ must be replaced by $\Pi_0 f$.
\end{proof}

We pass to kernel associations.
\begin{proof}[Proof of Proposition \ref{prop:assoc}]
The arguments we shall give go essentially as follows.
If $T$ is one of the scalar-valued operators and $K(\theta,\varphi)$ is a candidate for an associated
kernel, then for density reasons it is enough to verify that
$$
\langle Tf,g\rangle_{d\m_{\ab}} = \iint_{(0,\pi)^2} K(\theta,\varphi) f(\varphi) 
	\overline{g(\theta)} \, d\m_{\ab}(\varphi)d\m_{\ab}(\theta)
$$
for all $f,g \in C_c^{\infty}(0,\pi)$ with disjoint supports.
The definition of $T$ in $L^2(d\m_{\ab})$ by means of the spectral series together with
Parseval's identity allows us to write the left-hand side here as a series involving the Fourier-Jacobi
coefficients of $f$ and $g$ and in some cases also the auxiliary Jacobi systems \eqref{adorth}. 
It is sufficient to check that this series coincides with the right-hand side. 
This is clear on the formal level, after expressing the kernel $K(\theta,\varphi)$
in terms of a series involving products $\P_n^{\ab}(\theta)\P_n^{\ab}(\varphi)$ and then changing
orders of summation, integration and possibly differentiation. 
However, ensuring that these order changes are indeed legitimate is a delicate matter since  
the kernel has a non-integrable singularity. To perform this task, one has to use the fact that the
supports of $f$ and $g$ are disjoint, in order to avoid the singularity. As usual in similar situations,
this is combined with, among other things, estimates of expressions related to
the kernel and some information on the growth of the eigenfunctions $\P_n^{\ab}$ as $n \to \infty$.
The case of a vector-valued $T$ is in principle similar, only the technicalities are a bit more complex.
If $K(\theta,\varphi)$ is now a candidate for a $\mathbb{B}$-valued kernel associated with 
one of our square function operators, say $T$, then the task is easily reduced to verifying that
$$
\langle Tf,h\rangle = 
	\bigg\langle \int_{(0,\pi)} K(\cdot,\varphi) f(\varphi) \, d\m_{\ab}(\varphi), h \bigg\rangle
$$
for each fixed $f \in C_c^{\infty}(0,\pi)$ and a set of $h$ that spans a dense subspace of
$L^2_{\mathbb{B}^*}((\support f)^c,d\m_{\ab})$. Here the dual $\mathbb{B}^*$ is identified with
$\mathbb{B}$. The pairing above is understood in the sense of $L^2_{\mathbb{B}}((\support f)^c,d\m_{\ab})$
and its dual, which happens to be the same space in view of self-duality of $\mathbb{B}$.
From here, roughly speaking, one proceeds with manipulations, see \cite{StTo2,Szarek},
in the spirit described above for the scalar-valued case.
The case of the maximal operator is even easier, since then it is enough to test the identity
\eqref{id_CZ}
for $f \in C_c^{\infty}(0,\pi)$ and only by pairing with point measures 
$\delta_{t_0}\in \mathbb{B}^{*}$ at $t_0>0$.

All the relevant arguments needed to prove Proposition \ref{prop:assoc}
were given in detail elsewhere in the settings of Hermite and Laguerre
function expansions, see \cite{StTo,StTo2,StTo3,NS,Szarek}. 
Since the reasoning in the Jacobi setting is completely analogous, we only indicate what ingredients
specific to the present context are necessary to make the proofs go through.

To treat the imaginary powers $I_{\gamma}^{\ab}$, we proceed as in the proof of 
\cite[Proposition 4.2]{StTo3}, with the aid of \eqref{estjac}; notice that \eqref{estjac} 
implies immediately \eqref{growthFJ} specified to $p=1$ and $w\equiv 1$. 
The argument starts with an integration by parts in \eqref{imag_ker}, and to see that there will be 
no integrated term, one needs to know that the Jacobi-Poisson kernel has limit $0$ as either $t\to 0$ 
or $t\to\infty$. This, however, follows from Proposition \ref{prop:rep_H} below;
the decay at infinity is also visible in \eqref{PJser}.
The relevant estimate for functions $f,g\in C_c^{\infty}(0,\pi)$ with disjoint supports
$$
\iint_{(0,\pi)^2} \int_0^{\infty} \big|\partial_t H_t^{\ab}(\theta,\varphi)\big|\,dt \, 
	\big|\overline{g(\theta)}f(\varphi)|\, d\m_{\ab}(\theta) d\m_{\ab}(\varphi) < \infty
$$
holds because, for given compact and disjoint sets $E,F\subset (0,\pi)$, we have
$$
\int_0^{\infty} \big| \partial_t H_t^{\ab}(\theta,\varphi)\big|\, dt \lesssim 1, \qquad
	\theta \in E, \quad \varphi \in F.
$$
The last bound, in turn, can be easily justified by means of the technique developed in
Section~\ref{sec:ker}, see the proof of the growth condition in the case of $g_V^{\ab}$.

Considering the Riesz-Jacobi transforms $R_N^{\ab}$, we copy with appropriate adjustments the
reasoning from the proofs of \cite[Propositions 3.3 and 3.7]{NS}. The relevant ingredients
are the orthogonal decomposition of $\delta^N\P_n^{\ab}$ stated in Lemma \ref{decomp}, the estimate
\eqref{estjac} and a strengthened version of the growth condition for $R_N^{\ab}(\theta,\varphi)$,
$$
\int_0^{\infty} \big|\partial_{\theta}^N H_t^{\ab}(\theta,\varphi)\big| t^{N-1}\, dt \lesssim
	\frac{1}{\m_{\ab}(B(\theta,|\theta-\varphi|))}, \qquad \theta,\varphi \in (0,\pi).
$$
The last bound is proved implicitly in Section \ref{sec:ker}; see the proof of the growth condition
for the kernel associated to $R_N^{\ab}$.

To deal with the maximal operator $\H_{*}^{\ab}$, we first ensure that the vector-valued linear operator
$\mathcal{H}^{\ab}$ defined on $L^2(d\m_{\ab})$ and assigning to an $f\in L^2(d\m_{\ab})$ the function
$\H^{\ab}f$ whose value at $\theta \in (0,\pi)$ is
$$
\mathcal{H}^{\ab}f(\theta) = \{\H_t^{\ab}f(\theta)\}_{t>0},
$$
has indeed its values in the Bochner-Lebesgue space $L^2_{\mathbb{X}}(d\m_{\ab})$.
This, however, follows as in the proof of \cite[Theorem 2.1]{NS}, since $\H_t^{\ab}f(\theta)$
is continuous in $(t,\theta)\in (0,\infty)\times (0,\pi)$ for $f\in L^2(d\m_{\ab})$, and the
scalar-valued maximal operator $\H_*^{\ab}$ is bounded on $L^2(d\m_{\ab})$, see Proposition
\ref{prop:L2}. Indeed, given $f\in L^2(d\m_{\ab})$, the boundedness of the maximal
operator together with the completeness of $\{\P_n^{\ab}:n\ge 0\}$ in $L^2(d\m_{\ab})$
implies by standard arguments the existence of the limit $\lim_{t\to 0^+}\H_t^{\ab}f(\theta)$
for a.a. $\theta \in (0,\pi)$. The existence of $\lim_{t\to\infty}\H_t^{\ab}f(\theta)$ is more elementary
since by \eqref{estjac}
$$
|\H_t^{\ab}f(\theta)| \lesssim \bigg| \sum_{n\ge 0} e^{-t|n+\frac{\alpha+\beta+1}2|}
	\langle f, \P_n^{\ab}\rangle_{d\m_{\ab}} \P_n^{\ab}(\theta) \bigg|
	\lesssim \sum_{n\ge 0} e^{-t|n+\frac{\alpha+\beta+1}2|} (n+1)^{2\alpha+2\beta+5},
$$
which justifies the case when $\alpha+\beta+1 \neq 0$. When $\alpha+\beta+1=0$, we write
$$
\bigg|\H_t^{\ab}f(\theta)-\frac{1}{\m_{\ab}((0,\pi))}\int_{(0,\pi)}f(\theta)\, d\m_{\ab}(\theta)\bigg|
 \lesssim \sum_{n\ge 1} e^{-t|n+\frac{\alpha+\beta+1}2|} (n+1)^{2\alpha+2\beta+5}.
$$
From this point we continue using the arguments from the proof of \cite[Theorem 2.1]{NS} combined with 
the growth condition for $\{H_t^{\ab}(\theta,\varphi)\}_{t>0}$ proved in Section \ref{sec:ker}.

Finally, the treatment of the square functions $g_{M,N}^{\ab}$, $M,N\ge 0$, $M+N>0$, relies on repeating,
with suitable modifications, the arguments from the proof of \cite[Proposition 2.3]{Szarek}.
Here the important ingredients are: the estimate \eqref{estjac}, the $L^2$-boundedness
from Proposition \ref{prop:L2} and the growth condition for the associated vector-valued kernels 
proved in Section \ref{sec:ker}. In addition, in the cases of $g$-functions with non-trivial horizontal
component, the decomposition of $\delta^N\P_n^{\ab}$ from Lemma \ref{decomp} is needed.
\end{proof}

We remark that the proof just given is based on known arguments in the Laguerre setting,
even though the Hermite setting is the basic prototype, cf. \cite{StTo,StTo2,StTo3}.
This is because the Laguerre setting is closer to the present Jacobi context, sharing phenomena
absent in the Hermite case like the presence of the type parameters and the additional orthogonal
systems emerging from the decomposition \eqref{dec}.

\section{Kernel estimates} \label{sec:ker}

This section is devoted to proving all the necessary kernel estimates.
We start by deriving a suitable representation for the Jacobi-Poisson kernel \eqref{PJser}.
This will be achieved by applying the product formula due to Dijksma and Koornwinder \cite{DK},
\begin{align*}
& P_n^{\ab}(1-2s^2)P_n^{\ab}(1-2t^2) = \frac{\Gamma(\alpha+\beta+1)\Gamma(n+\alpha+1)\Gamma(n+\beta+1)}
	{\pi n! \Gamma(n+\alpha+\beta+1)\Gamma(\alpha+1\slash 2)\Gamma(\beta+1\slash 2)} \\
& \quad \times \int_{-1}^1\int_{-1}^1 C_{2n}^{\alpha+\beta+1}\big(ust+v\sqrt{1-s^2}\sqrt{1-t^2}\big)
	(1-u^2)^{\alpha-1\slash 2}(1-v^2)^{\beta-1\slash 2} du dv,
\end{align*}
valid for $\alpha,\beta>-1\slash 2$; here $C_k^{\lambda}$ is the classical $k$th Gegenbauer
polynomial of type $\lambda$. Let $\Pi_{\alpha}$ be the probability measure on the interval $[-1,1]$
defined for $\alpha>-1\slash 2$ by
$$
d\Pi_{\alpha}(u) = \frac{\Gamma(\alpha+1)}{\sqrt{\pi}\Gamma(\alpha+1\slash 2)}
	(1-u^2)^{\alpha-1\slash 2} du.
$$
In the limit case $\alpha=-1\slash 2$, we put
$$
\Pi_{-1\slash 2} = \frac{1}{2}(\delta_{-1}+\delta_{1}).
$$
Note that $\Pi_{-1\slash 2}$ is the weak limit of $\Pi_{\alpha}$ as $\alpha \to -1\slash 2$.
Using the above product formula with $s=\sin\frac{\theta}2$ and $t=\sin\frac{\varphi}2$,
the relation between the polynomials $P_n^{\ab}$ and $\P_n^{\ab}$, and the fact that
$\m_{\ab}(0,\pi)=\Gamma(\alpha+1)\Gamma(\beta+1)\slash \Gamma(\alpha+\beta+2)$, we arrive at the identity
\begin{align*}
 \P_n^{\ab}(\theta)\P_n^{\ab}(\varphi) & = \frac{1}{\m_{\ab}(0,\pi)} 
	\frac{2n+\alpha+\beta+1}{\alpha+\beta+1} \\ &\quad \times \iint d\Pi_{\alpha}(u)d\Pi_{\beta}(v)
	\,C_{2n}^{\alpha+\beta+1}\Big( u \sin\frac{\theta}2\sin\frac{\varphi}2 + v \cos\frac{\theta}2
		\cos\frac{\varphi}2 \Big).
\end{align*}
Thus, letting $\lambda=\alpha+\beta+1$, we have in view of \eqref{PJser}
\begin{align*}
H_t^{\ab}(\theta,\varphi) & =  \frac{1}{\m_{\ab}(0,\pi)} \iint 
	d\Pi_{\alpha}(u)d\Pi_{\beta}(v) \\
& \quad \times \sum_{n=0}^{\infty} e^{-\frac{t}2(2n+\lambda)} \frac{2n+\lambda}{\lambda}
	C_{2n}^{\lambda}\Big( u \sin\frac{\theta}2\sin\frac{\varphi}2 + v \cos\frac{\theta}2
		\cos\frac{\varphi}2 \Big).
\end{align*}
To sum the last series, we use the generating function (cf. \cite[(1.27)]{A})
\begin{equation} \label{Cgen}
\sum_{n=0}^{\infty} \frac{n+\lambda}{\lambda} C_n^{\lambda}(z)r^n =
	\frac{1-r^2}{(1-2zr+r^2)^{\lambda+1}}, \qquad |r|<1, \quad \lambda>0.
\end{equation}
The fact that Gegenbauer polynomials of even orders are even functions, and those of odd orders are
odd functions, reveals that summing only over even indices in \eqref{Cgen}
will produce the even part of the right-hand side. Therefore we get
\begin{align*}
& H_t^{\ab}(\theta,\varphi) = \frac{1}{2^{\lambda+1}\m_{\ab}(0,\pi)} \sinh\frac{t}2
\iint d\Pi_{\alpha}(u)d\Pi_{\beta}(v) \\ 
&  \times \!\Bigg[ \frac{1}{(\cosh\frac{t}2-u \sin\frac{\theta}2\sin\frac{\varphi}2 - v \cos\frac{\theta}2
		\cos\frac{\varphi}2)^{\lambda+1}} + \frac{1}{(\cosh\frac{t}2+u \sin\frac{\theta}2
		\sin\frac{\varphi}2 + v \cos\frac{\theta}2 \cos\frac{\varphi}2)^{\lambda+1}}\Bigg].
\end{align*}
Taking into account the symmetry of the measures $\Pi_{\alpha}$ and $\Pi_{\beta}$,
we end up with the formula
$$
H_t^{\ab}(\theta,\varphi) = \frac{\sinh(t\slash 2)}{2^{\alpha+\beta+1}\m_{\ab}(0,\pi)} 
\iint \frac{d\Pi_{\alpha}(u)d\Pi_{\beta}(v)}
	{(\cosh\frac{t}2-u \sin\frac{\theta}2\sin\frac{\varphi}2 - v \cos\frac{\theta}2
		\cos\frac{\varphi}2)^{\alpha+\beta+2}}.
$$
By continuity arguments, this representation remains valid in the limiting cases when
$\alpha=-1\slash 2$ or $\beta=-1\slash 2$. In particular, for $\alpha=\beta=-1\slash 2$ the formula gives
$$
H_t^{-1\slash 2,-1\slash 2}(\theta,\varphi) = 
	\frac{1}{2\pi} \bigg[ \frac{\sinh t}{\cosh t - \cos(\theta-\varphi)}
	+ \frac{\sinh t}{\cosh t - \cos(\theta+\varphi)} \bigg].
$$
Here one recovers the standard Poisson kernel of the unit disc, applied to even functions
on the boundary, since the last expression equals 
$(P(re^{i\theta},e^{i\varphi})+P(re^{i\theta},e^{-i\varphi}))\slash 2$, with $r=e^{-t}$
and $P(z,w) = (2\pi)^{-1}(1-|z|^2)\slash |z-w|^2$. 

This provides a symmetric and nonnegative expression for $H_t^{\ab}(\theta,\varphi)$,
which turns out to be especially well suited to our framework. For the considerations that follow, it is
convenient to rewrite the last expression for $H_t^{\ab}$ in terms of the function
$$
q(\theta,\varphi,u,v) = 1 - u \sin\frac{\theta}2 \sin\frac{\varphi}2
	- v \cos\frac{\theta}2 \cos\frac{\varphi}2, \qquad \theta,\varphi \in (0,\pi), \quad u,v \in [-1,1].
$$
\begin{prop} \label{prop:rep_H}
The Jacobi-Poisson kernel can be written as
\begin{equation} \label{PJker}
H_t^{\ab}(\theta,\varphi) = c_{\alpha,\beta} \,\sinh\frac{t}2
	\iint \frac{d\Pi_{\alpha}(u)d\Pi_{\beta}(v)}
	{(\cosh\frac{t}2-1 + q(\theta,\varphi,u,v))^{\alpha+\beta+2}},
\end{equation}
with $c_{\alpha,\beta} =2^{-\alpha-\beta-1}\slash \m_{\ab}(0,\pi)$. \qed
\end{prop} 
For further reference, observe that
\begin{equation} \label{stpg10}
2\sin^2\frac{\theta-\varphi}4 = q(\theta,\varphi,1,1) \le q(\theta,\varphi,u,v) \le
	q(\theta,\varphi,-1,-1) = 2\cos^2\frac{\theta-\varphi}4,
\end{equation}
so the quantity $q(\theta,\varphi,u,v)$ is nonnegative 
and bounded from above by $2$. Moreover, it is strictly positive when $\theta \neq \varphi$.

The following lemma describes the measure of the interval $B(\theta,|\varphi-\theta|)$ 
and is valid for all $\alpha,\beta > -1$. 
\begin{lem} \label{ball}
For all $\theta,\varphi \in (0,\pi)$, one has 
$$
\m_{\ab}\big( B(\theta,|\varphi-\theta|)\big) \simeq |\varphi-\theta| (\theta+\varphi)^{2\alpha+1}
	(\pi-\theta+\pi-\varphi)^{2\beta+1}.
$$
\end{lem}

\begin{proof}
Simple exercise.
\end{proof}

The lemma below establishes an important connection between estimates naturally emerging 
from the representation \eqref{PJker} and the standard estimates related to the space of 
homogeneous type $((0,\pi),d\m_{\ab},|\cdot|)$. This is the essence of the whole technique.
A similar result, with appropriate adjustments, holds also in a multi-dimensional setting. 
\begin{lem} \label{bridge}
Let $\alpha,\beta \ge -1\slash 2$. Then
\begin{align*}
\iint \frac{d\Pi_{\alpha}(u)d\Pi_{\beta}(v)}{(q(\theta,\varphi,u,v))^{\alpha+\beta+3\slash 2}}
	& \lesssim \frac{1}{\m_{\ab}(B(\theta,|\varphi-\theta|))}, \qquad \theta,\varphi \in (0,\pi), 
	\quad \theta\neq \varphi,\\
	\iint \frac{d\Pi_{\alpha}(u)d\Pi_{\beta}(v)}{(q(\theta,\varphi,u,v))^{\alpha+\beta+2}} & \lesssim
	\frac{1}{|\theta-\varphi|\m_{\ab}(B(\theta,|\varphi-\theta|))}, 
	\qquad \theta,\varphi \in (0,\pi), \quad \theta\neq \varphi.
\end{align*}
\end{lem}

To prove this we will need the following simple estimate, see \cite[Lemma 5.8]{NS}.
\begin{lem} \label{lem58}
Let $\gamma \ge -1\slash 2$ and $\lambda > 0$ be fixed. Then
$$
\int  
	\frac{d\Pi_{\gamma}(s)}{(A-Bs)^{\gamma+1\slash 2 + \lambda}}
	\lesssim \frac{1}{A^{\gamma+1\slash 2}(A-B)^{\lambda}}, \qquad A>B>0.
$$
\end{lem}

\begin{proof}[Proof of Lemma \ref{bridge}]
Applying Lemma \ref{lem58} first to the integral against $d\Pi_{\beta}(v)$ and then again to the
integral against $d\Pi_{\alpha}(u)$ we obtain
\begin{align*}
& \iint \frac{d\Pi_{\alpha}(u)d\Pi_{\beta}(v)}{(q(\theta,\varphi,u,v))^{\alpha+\beta+3\slash 2}}\\
& \lesssim \int \frac{d\Pi_{\alpha}(u)}
	{(1-u\sin\frac{\theta}2\sin\frac{\varphi}2)^{\beta+1\slash 2} 
	(1-\cos\frac{\theta}2\cos\frac{\varphi}2-u\sin\frac{\theta}2\sin\frac{\varphi}2)^{\alpha+1}} \\
& \le \frac{1}{(1-\sin\frac{\theta}2\sin\frac{\varphi}2)^{\beta+1\slash 2}}
	\int \frac{d\Pi_{\alpha}(u)}
	{(1-\cos\frac{\theta}2\cos\frac{\varphi}2-u\sin\frac{\theta}2\sin\frac{\varphi}2)^{\alpha+1}} \\
& \lesssim \frac{1}{(1-\sin\frac{\theta}2\sin\frac{\varphi}2)^{\beta+1\slash 2}
	(1-\cos\frac{\theta}2\cos\frac{\varphi}2)^{\alpha+1\slash 2}
	(1-\cos\frac{\theta}2\cos\frac{\varphi}2 -\sin\frac{\theta}2\sin\frac{\varphi}2)^{1\slash 2}}.
\end{align*}
We now observe that
\begin{align*}
1- \cos\frac{\theta}{2} \cos\frac{\varphi}{2} & = \sin^2\frac{\theta-\varphi}{4}
	+ \sin^2\frac{\theta+\varphi}{4} \simeq \theta^2 + \varphi^2, \\
1 - \sin\frac{\theta}{2}\sin\frac{\varphi}{2} & = \sin^2\frac{\theta-\varphi}{4}
	+ \cos^2\frac{\theta+\varphi}{4} \simeq (\pi-\theta)^2 + (\pi - \varphi)^2, \\
1- \cos\frac{\theta}{2} \cos\frac{\varphi}{2} - \sin\frac{\theta}{2}\sin\frac{\varphi}{2} & 
 = 1 - \cos\frac{\theta-\varphi}{2} = 2\sin^2\frac{\theta-\varphi}{4} \simeq (\theta-\varphi)^2, 	
\end{align*}
which gives
$$
\iint \frac{d\Pi_{\alpha}(u)d\Pi_{\beta}(v)}{(q(\theta,\varphi,u,v))^{\alpha+\beta+3\slash 2}}
\lesssim \frac{1}{|\theta-\varphi|(\theta^2+\varphi^2)^{\alpha+1\slash 2}
	((\pi-\theta)^2+(\pi-\varphi)^2)^{\beta+1\slash 2}}, 
$$
uniformly in $\theta,\varphi \in (0,\pi)$. This, in view of Lemma \ref{ball}, implies the 
first estimate of the lemma.

Parallel arguments justify the remaining estimate.
\end{proof}

The following technical result will be frequently applied in the proofs of the kernel estimates.
\begin{lem} \label{trig}
For all $\theta,\varphi \in (0,\pi)$ and $u,v \in [-1,1]$, one has
$$
\big| \partial_{\theta} q(\theta,\varphi,u,v) \big| \lesssim \sqrt{q(\theta,\varphi,u,v)} 
\qquad \textrm{and} \qquad
\big| \partial_{\varphi} q(\theta,\varphi,u,v) \big| \lesssim \sqrt{q(\theta,\varphi,u,v)}.
$$
\end{lem}

\begin{proof}
It is sufficient to show the first inequality, since the second then follows from the symmetry
$q(\theta,\varphi,u,v)=q(\varphi,\theta,u,v)$. Observe that
\begin{equation} \label{stst}
q(\theta,\varphi,u,v) = 1 - \cos\frac{\theta-\varphi}2 + (1-u)\sin\frac{\theta}2\sin\frac{\varphi}2
	+(1-v)\cos\frac{\theta}2\cos\frac{\varphi}2.
\end{equation}
Thus we have
\begin{align*}
|\partial_{\theta}q(\theta,\varphi,u,v)| & = \frac{1}{2} \Big| \sin\frac{\theta-\varphi}2
	+ (1-u)\cos\frac{\theta}2\sin\frac{\varphi}2 - (1-v)\sin\frac{\theta}2\cos\frac{\varphi}2 \Big| \\
& \le |\theta - \varphi| + (1-u)\varphi + (1-v)(\pi-\varphi).
\end{align*}
In the last expression, the first term is controlled by $\sqrt{q(\theta,\varphi,u,v)}$ because,
in view of \eqref{stst}, $q(\theta,\varphi,u,v) \gtrsim (\theta-\varphi)^2$.
For the second term, we write
\begin{align*}
(1-u)^2\varphi^2 & \le (1-u)^2(\theta^2+ \varphi^2) = (1-u)^2(\theta-\varphi)^2 + 2(1-u)^2\theta\varphi \\
& \lesssim (\theta-\varphi)^2 + (1-u) \sin\frac{\theta}2 \sin\frac{\varphi}2 \\
& \lesssim q(\theta,\varphi,u,v).
\end{align*}
A reflection of this argument in $\pi\slash 2$ covers the third term.
\end{proof}
We remark that $1\slash \sqrt{2}$ is the optimal constant for the inequalities in Lemma 
\ref{trig}, but proving this requires a more detailed analysis.

The result below will come into play when we verify the smoothness estimates \eqref{sm1} and \eqref{sm2}
for the relevant vector-valued kernels. It will enable us to reduce the difference conditions
to certain gradient estimates, which are easier to verify.
\begin{lem} \label{lem:comp}
For all $\theta,\widetilde{\theta},\varphi \in (0,\pi)$ with 
$|\theta-\varphi|>2|\theta-\widetilde{\theta}|$ and all $u,v \in [-1,1]$,
$$
q(\theta,\varphi,u,v) \simeq q(\widetilde{\theta},\varphi,u,v).
$$
Similarly, for all $\theta,\varphi,\widetilde{\varphi} \in (0,\pi)$ with 
$|\theta-\varphi|>2|\varphi-\widetilde{\varphi}|$ and all $u,v \in [-1,1]$,
$$
q(\theta,\varphi,u,v) \simeq q(\theta,\widetilde{\varphi},u,v).
$$
\end{lem}

\begin{proof}
For symmetry reasons, it is enough to verify the first relation. By \eqref{stst},
\begin{equation} \label{ststst}
q(\theta,\varphi,u,v) \simeq (\theta-\varphi)^2 + (1-u)\theta\varphi + (1-v)(\pi-\theta)(\pi-\varphi).
\end{equation}
The three terms in the expression \eqref{ststst} together determine the order of magnitude of
$q(\theta,\varphi,u,v)$.
When $\theta$ is replaced by $\widetilde{\theta}$, the first term does not change its order of magnitude,
because of the hypothesis made. To deal with the second term, we first assume that $\varphi < 2\theta$.
Then the hypothesis implies $|\theta-\widetilde{\theta}|<|\theta-\varphi|\slash 2< \theta\slash 2$.
Thus $\theta \simeq \widetilde{\theta}$, which means that the replacement does not change the order
of magnitude of the second term. In the remaining case $\varphi \ge 2 \theta$, we have
$\varphi \simeq |\varphi-\theta| \simeq |\varphi-\widetilde{\theta}|$ and $\theta \le |\varphi-\theta|$,
so that $\widetilde{\theta}< \theta + |\varphi-\theta|\slash 2 \lesssim |\varphi-\widetilde{\theta}|$.
Then the second term is dominated by the first, in \eqref{ststst} and in the analogous expression
with $\widetilde{\theta}$ instead of $\theta$.
Since the third term can be treated like the second after a reflection in $\pi\slash 2$, the lemma follows.
\end{proof}

In the sequel we will often omit the arguments and write $\q$ instead of $q(\theta,\varphi,u,v)$.
We shall tacitly assume that passing with the differentiation in $\theta$, or $\varphi$ or $t$,
under the integral against $d\Pi_{\alpha}(u)\, d\Pi_{\beta}(v)$ or against $dt$ is legitimate;
similarly for changing orders of integrals or sums. This is indeed the case in all the relevant cases,
which may be verified in a straightforward manner by means of the estimates obtained in the proof of
Theorem \ref{thm:stand}; see \cite[Section 5]{NS} or \cite[Section 4]{Szarek} where the details are 
given in the context of Laguerre function expansions.

We are now in a position to give the proof of Theorem \ref{thm:stand}.
We first treat the kernel $\{H_t^{\ab}(\theta,\varphi)\}_{t>0}$ associated to the Jacobi-Poisson
semigroup maximal operator, which is the easiest to estimate.

\begin{proof}[Proof of Theorem \ref{thm:stand}; the case of $\H_*^{\ab}$]
We first deal with the growth condition \eqref{gr} specified to $\mathbb{B}=\mathbb{X}$. Observe that
$$
\frac{\sinh\frac{t}2}{(\cosh\frac{t}2-1+\q)^{\alpha+\beta+2}} \lesssim
	\frac{1}{\q^{\alpha+\beta+3\slash 2}},
$$
uniformly in $\q$ and $t>0$. For $t$ small this follows by the asymptotics 
$\cosh\frac{t}2-1=\mathcal{O}(t^2)$, $t\to 0$, and for large $t$ we 
use the asymptotics $\tanh\frac{t}2 = \mathcal{O}(1)$, $t\to\infty$,
and boundedness of the quantity $\q$. Then Proposition \ref{prop:rep_H} implies
$$
\|\{H_t^{\ab}(\theta,\varphi)\}\|_{\mathbb{X}} \lesssim 
	\iint \frac{d\Pi_{\alpha}(u)d\Pi_{\beta}(v)}
		{\q^{\alpha+\beta+3\slash 2}},
$$
and the growth bound follows immediately from Lemma \ref{bridge}.

To show the smoothness conditions \eqref{sm1} and \eqref{sm2}, it is enough to consider \eqref{sm1}, 
by symmetry. We first analyze the derivative $\partial_{\theta}H_t^{\ab}(\theta,\varphi)$.
Applying Proposition \ref{prop:rep_H} and Lemma \ref{trig}, we get
\begin{align*}
\Big| \frac{\partial}{\partial \theta} H_t^{\ab}(\theta,\varphi)\Big| & \lesssim
	\sinh\frac{t}2 \iint {d\Pi_{\alpha}(u)d\Pi_{\beta}(v)} 
		\frac{|\partial_{\theta}\q|}{(\cosh\frac{t}2-1+\q)^{\alpha+\beta+3}} \\
& \lesssim \sinh\frac{t}2 \iint \frac{d\Pi_{\alpha}(u)d\Pi_{\beta}(v)} 
		{(\cosh\frac{t}2-1+\q)^{\alpha+\beta+5\slash 2}}.
\end{align*}
By the Mean Value Theorem and the above estimate, we have
\begin{align*}
|H_t^{\ab}(\theta,\varphi)-H_t^{\ab}(\theta',\varphi)| & \le |\theta-\theta'|
	|\partial_{\theta}H_t^{\ab}(\widetilde{\theta},\varphi)| \\
& \lesssim  |\theta-\theta'| \sinh\frac{t}2 \iint \frac{d\Pi_{\alpha}(u)d\Pi_{\beta}(v)} 
	{(\cosh\frac{t}2-1+q(\widetilde{\theta},\varphi,u,v))^{\alpha+\beta+5\slash 2}},
\end{align*}
where $\widetilde{\theta}$ is a convex combination of $\theta$ and $\theta'$ 
(notice that $\widetilde{\theta}$ depends also on $t$). Then assuming that
$|\theta-\varphi|>2|\theta-\theta'|$, which implies $|\theta-\varphi|>2|\theta-\widetilde{\theta}|$,
and using Lemma \ref{lem:comp}, we get that 
$$
|H_t^{\ab}(\theta,\varphi)-H_t^{\ab}(\theta',\varphi)| \lesssim 
	|\theta-\theta'| \sinh\frac{t}2 \iint \frac{d\Pi_{\alpha}(u)d\Pi_{\beta}(v)} 
	{(\cosh\frac{t}2-1+\q)^{\alpha+\beta+5\slash 2}}.
$$
Now the conclusion follows by Lemma \ref{bridge} as in case of the growth estimate.
\end{proof}

We next show that the kernels associated to the imaginary powers of the Jacobi operator
satisfy the standard estimates for $\alpha,\beta \ge -1\slash 2$ such that $\alpha+\beta > -1$
(the case $\alpha=\beta=-1\slash 2$ must be excluded since then $0$ 
is an eigenvalue of the Jacobi operator). Recall that
$$
K^{\ab}_{\gamma}(\theta,\varphi) = \frac{1}{\Gamma(2i\gamma)} \int_0^{\infty}
	H_t^{\ab}(\theta,\varphi) t^{2i\gamma-1}dt, \qquad \gamma \in \R\backslash\{0\}.
$$

\begin{proof}[Proof of Theorem \ref{thm:stand}; the case of $I_{\gamma}^{\ab}$]
By Proposition \ref{prop:rep_H},
$$
|K^{\ab}_{\gamma}(\theta,\varphi)| \lesssim \int_0^{\infty} \frac{1}{t} \sinh\frac{t}2
\iint \frac{d\Pi_{\alpha}(u)d\Pi_{\beta}(v)}
	{(\cosh\frac{t}2 - 1 + \q)^{\alpha+\beta+2}}\, dt.
$$
We now split the integral in $t$ into the intervals $(0,1)$ and $(1,\infty)$ and denote the resulting
integrals by $I_0$ and $I_{\infty}$, respectively. Then
$$
I_{0} \lesssim \iint {d\Pi_{\alpha}(u)d\Pi_{\beta}(v)}
	\int_0^1 \frac{dt}{(t^2 + \q)^{\alpha+\beta+2}}\, dt
$$
and changing the variable $t\mapsto \sqrt{\q}s$ we get
$$
I_0 \lesssim \iint \frac{d\Pi_{\alpha}(u)d\Pi_{\beta}(v)}
	{\q^{\alpha+\beta+3\slash 2}} \int_0^{\infty} \frac{ds}{(1+s^2)^{\alpha+\beta+2}}
\lesssim \iint \frac{d\Pi_{\alpha}(u)d\Pi_{\beta}(v)}
	{\q^{\alpha+\beta+3\slash 2}}.
$$
Estimating $I_{\infty}$ is even more straightforward; we have
$$
I_{\infty} \lesssim \iint {d\Pi_{\alpha}(u)d\Pi_{\beta}(v)} \int_1^{\infty}
	\frac{e^{t\slash 2}dt}{(e^{t\slash 2})^{\alpha+\beta+2}}
	\lesssim \iint {d\Pi_{\alpha}(u)d\Pi_{\beta}(v)}
	\lesssim \iint \frac{d\Pi_{\alpha}(u)d\Pi_{\beta}(v)}
	{\q^{\alpha+\beta+3\slash 2}},
$$
where in the last step we used the boundedness of $\q$. In view of Lemma \ref{bridge}, the growth 
condition \eqref{gr} (with $\mathbb{B}=\mathbb{C}$) for $K^{\ab}_{\gamma}(\theta,\varphi)$ follows.

To show the gradient condition \eqref{sm}, we use analogous arguments combined with Lemma \ref{trig}.
For symmetry reasons, we may consider only the partial derivative in $\theta$. Then
$$
\Big|\frac{\partial}{\partial \theta} K^{\ab}_{\gamma}(\theta,\varphi)\Big| \lesssim
\int_0^{\infty} \frac{1}{t} \sinh\frac{t}2 \iint {d\Pi_{\alpha}(u)d\Pi_{\beta}(v)}
	\frac{|\partial_{\theta}\q|}{(\cosh\frac{t}2 - 1 + \q)^{\alpha+\beta+3}}\, dt.
$$
Applying Lemma \ref{trig} and proceeding as before, we get
$$
\Big|\frac{\partial}{\partial \theta} K^{\ab}_{\gamma}(\theta,\varphi)\Big|  \lesssim
\int_0^{\infty} \frac{1}{t} \sinh\frac{t}2 \iint \frac{d\Pi_{\alpha}(u)d\Pi_{\beta}(v)}
	{(\cosh\frac{t}2 - 1 + \q)^{\alpha+\beta+5\slash 2}}\, dt 
 \lesssim \iint \frac{d\Pi_{\alpha}(u)d\Pi_{\beta}(v)}{\q^{\alpha+\beta+2}}.
$$
The desired conclusion follows now from Lemma \ref{bridge}.
\end{proof}

The next kernels to be considered are those of the Riesz-Jacobi transforms of arbitrary order $N$,
$$
R_N^{\ab}(\theta,\varphi) = \frac{1}{\Gamma(N)}\int_0^{\infty} 
	\partial_{\theta}^N H_t^{\ab}(\theta,\varphi) t^{N-1}\,dt, \qquad N\ge 1.
$$
However, for the sake of clarity and the reader's convenience, we first treat separately and in greater
detail the more elementary case of the Riesz-Jacobi transform of order $N=1$.
We will write simply $R^{\ab}(\theta,\varphi)$ instead of $R_1^{\ab}(\theta,\varphi)$.

\begin{proof}[Proof of Theorem \ref{thm:stand}; the case of $R_1^{\ab}$]
By an elementary computation and Lemma \ref{trig},
\begin{align*}
|R^{\ab}(\theta,\varphi)| & \lesssim \int_0^{\infty}  \sinh\frac{t}2
\iint {d\Pi_{\alpha}(u)d\Pi_{\beta}(v)} \frac{|\partial_{\theta}\q|}
	{(\cosh\frac{t}2 - 1 + \q)^{\alpha+\beta+3}}\, dt \\
& \lesssim \iint {d\Pi_{\alpha}(u)d\Pi_{\beta}(v)} \int_0^{\infty}
	\frac{\sinh\frac{t}2 \, dt}{(\cosh\frac{t}2 - 1 + \q)^{\alpha+\beta+5\slash 2}}.
\end{align*}
From here we proceed as in the case of $I_{\gamma}^{\ab}$. Splitting the integral in $t$ into 
$I_0$ and $I_{\infty}$, as before, we get
\begin{align*}
I_0 & \lesssim \iint {d\Pi_{\alpha}(u)d\Pi_{\beta}(v)} \int_0^1
	\frac{t \, dt}{(t^2 + \q)^{\alpha+\beta+5\slash 2}} 
	\le \iint \frac{d\Pi_{\alpha}(u)d\Pi_{\beta}(v)}{\q^{\alpha+\beta+3\slash 2}} \int_0^{\infty} 
		\frac{s\,ds}{(1+s^2)^{\alpha+\beta+5\slash 2}}, \\
I_{\infty} & \lesssim \iint {d\Pi_{\alpha}(u)d\Pi_{\beta}(v)} \int_1^{\infty}
	\frac{e^{t\slash 2} \, dt}{(e^{t\slash 2})^{\alpha+\beta+5\slash 2}} \lesssim
	\iint \frac{d\Pi_{\alpha}(u)d\Pi_{\beta}(v)}{\q^{\alpha+\beta+3\slash 2}},
\end{align*}
where in the last estimate we used the boundedness of $\q$. Thus
$$
|R^{\ab}(\theta,\varphi)| \lesssim \iint \frac{d\Pi_{\alpha}(u)d\Pi_{\beta}(v)}
		{\q^{\alpha+\beta+3\slash 2}},
$$
and the asserted growth condition \eqref{gr} (with $\mathbb{B}=\mathbb{C}$) follows from Lemma \ref{bridge}.

We pass to the smoothness condition \eqref{sm} and start by finding bounds for the relevant derivatives.
Observe that since $\partial_{\theta}^2 \q = (1-\q)\slash 4$,
\begin{align*}
\bigg| \frac{\partial}{\partial \theta} \bigg(
 \frac{\partial_{\theta}\q}{(\cosh\frac{t}2 - 1 + \q)^{\alpha+\beta+3}} \bigg) \bigg| & \lesssim
 \frac{(\partial_{\theta}\q)^2}{(\cosh\frac{t}2 - 1 + \q)^{\alpha+\beta+4}}
	+ \frac{|\q-1|}{(\cosh\frac{t}2 - 1 + \q)^{\alpha+\beta+3}} \\
& \lesssim \frac{1}{(\cosh\frac{t}2 - 1 + \q)^{\alpha+\beta+3}},
\end{align*}
where in the last step we used Lemma \ref{trig} and the boundedness of the quantity $\q$. 
Similarly, using this time both inequalities of Lemma \ref{trig},
\begin{align*}
\bigg| \frac{\partial}{\partial \varphi} \bigg(
 \frac{\partial_{\theta}\q}{(\cosh\frac{t}2 - 1 + \q)^{\alpha+\beta+3}} \bigg) \bigg|
& \lesssim \frac{|\partial_{\theta}\q \partial_{\varphi}\q|}
	{(\cosh\frac{t}2 - 1 + \q)^{\alpha+\beta+4}}
	+ \frac{|q(\theta,\varphi,v,u)-1|}{(\cosh\frac{t}2 - 1 + \q)^{\alpha+\beta+3}} \\
& \lesssim \frac{1}{(\cosh\frac{t}2 - 1 + \q)^{\alpha+\beta+3}}.
\end{align*}
Taking the above bounds and \eqref{PJker} into account, we see that
$$
\Big| \frac{\partial}{\partial \theta} R^{\ab}(\theta,\varphi) \Big| +
	\Big| \frac{\partial}{\partial \varphi} R^{\ab}(\theta,\varphi) \Big| 
\lesssim \iint {d\Pi_{\alpha}(u)d\Pi_{\beta}(v)} \int_0^{\infty}
	\frac{\sinh\frac{t}2 \, dt}{(\cosh\frac{t}2 - 1 + \q)^{\alpha+\beta+3}}.
$$
Arguing as in case of the growth condition, we infer that
$$
\Big| \frac{\partial}{\partial \theta} R^{\ab}(\theta,\varphi) \Big| +
	\Big| \frac{\partial}{\partial \varphi} R^{\ab}(\theta,\varphi) \Big|
\lesssim \iint \frac{d\Pi_{\alpha}(u)d\Pi_{\beta}(v)}
		{\q^{\alpha+\beta+2}},
$$
and this combined with Lemma \ref{bridge} leads to the desired conclusion.
\end{proof}

To estimate the kernel $R_N^{\ab}(\theta,\varphi)$ for a general $N\ge 1$, we will need 
the following technical result.

\begin{lem} \label{derb}
For $\alpha,\beta \ge -1\slash 2$ and $N = 0,1,2,\ldots$
\begin{align*}
\bigg| \partial_{\theta}^N {\Big(\cosh\frac{t}2-1+\q\Big)^{-\alpha-\beta-2}}\bigg| &
\lesssim 
\begin{cases}
{(\cosh\frac{t}2-1+\q)^{-\alpha-\beta-2-N\slash 2}}, & t \le 1\\
{(\cosh\frac{t}2-1+\q)^{-\alpha-\beta-5\slash 2}}, & t > 1, \; N \ge 1
\end{cases}, \\
\bigg| \partial_{\varphi}\partial_{\theta}^N {\Big(\cosh\frac{t}2-1+\q\Big)^{-\alpha-\beta-2}}\bigg| &
\lesssim 
\begin{cases}
{(\cosh\frac{t}2-1+\q)^{-\alpha-\beta-5\slash 2-N\slash 2}}, & t \le 1\\
{(\cosh\frac{t}2-1+\q)^{-\alpha-\beta-5\slash 2}}, & t > 1 
\end{cases}.
\end{align*}
\end{lem}

Clearly, this still holds if $\alpha+\beta$ is replaced in both sides of the inequalities
by any quantity $\gamma$ satisfying $\gamma \ge -1$.

\begin{proof}[Proof of Lemma \ref{derb}]
We assume $N \ge 1$. The simple case $N=0$ is left to the reader.
To analyze the relevant derivatives, we will use Fa\`a di Bruno's formula. 
Choosing $g(x)=x^{-\alpha-\beta-2}$ and $f(\theta)=\cosh\frac{t}2-1+\q$ in \eqref{Faa},
it follows that $\partial_{\theta}^N(\cosh\frac{t}2-1+\q)^{-\alpha-\beta-2}$ 
is a linear combination of expressions of the form
\begin{equation} \label{f1}
\frac{(\partial^1_{\theta}\q)^{k_1}\cdots (\partial^N_{\theta}\q)^{k_N}}
	{(\cosh\frac{t}2-1+\q)^{\alpha+\beta+2+k_1+\ldots+k_N}}, 
\end{equation}
where $k_1+2k_2+\ldots+Nk_N = N$. Since for $m\ge 1$
$$
\partial_{\theta}^{2m}\q = (-4)^{-m}(\q-1), \qquad \partial_{\theta}^{2m-1}\q = 
(-4)^{1-m}\partial_{\theta}\q,
$$
we see by Lemma \ref{trig} and the boundedness of $\q$ that for $m \ge 1$
$$
|\partial_{\theta}^m\q| \lesssim
\begin{cases}
1, & m \;\textrm{even}\\
\sqrt{\q}, & m \; \textrm{odd}
\end{cases}.
$$
This combined with \eqref{f1} implies
$$
\bigg| \partial_{\theta}^N {\Big(\cosh\frac{t}2-1+\q\Big)^{-\alpha-\beta-2}}\bigg|
\lesssim \sum_{k_1+2k_2+\ldots+Nk_n=N} \frac{\sqrt{\q}^{k_1+k_3+\ldots+k_{\tilde{N}}}}
	{(\cosh\frac{t}2-1+\q)^{\alpha+\beta+2+k_1+\ldots+k_N}}, 
$$
where $\tilde{N}=N$ if $N$ is odd and $\tilde{N}=N-1$ if $N$ is even. Taking into account the 
boundedness of $\q$ and observing that the constraint $k_1+2k_2+\ldots+Nk_n=N$ forces 
$k_1+\ldots+k_N - (k_1+k_3+\ldots+k_{\tilde{N}})\slash 2 \le N\slash 2$,
we get the first two estimates of the lemma.

Applying $\partial_{\varphi}$ to \eqref{f1}, we infer that 
$\partial_{\varphi}\partial_{\theta}^N(\cosh\frac{t}2-1+\q)^{-\alpha-\beta-2}$ is a linear
combination of expressions of the form
$$
\frac{(\partial^1_{\theta}\q)^{k_1}\cdots (\partial^N_{\theta}\q)^{k_N}\partial_{\varphi}\q}
	{(\cosh\frac{t}2-1+\q)^{\alpha+\beta+3+k_1+\ldots+k_N}} \quad \textrm{and} \quad
\frac{(\partial^1_{\theta}\q)^{k_1}\cdots (\partial^N_{\theta}\q)^{k_N}
		\partial_{\varphi}\partial_{\theta}^i\q\slash \partial_{\theta}^i\q}
	{(\cosh\frac{t}2-1+\q)^{\alpha+\beta+2+k_1+\ldots+k_N}}, 
$$
where $k_1+2k_2+\ldots+Nk_N = N$, $i=1,\ldots,N$; for the second form we exclude
the cases when $k_i=0$. 
By means of the bounds on $|\partial_{\theta}^m\q|$ and $|\partial_{\varphi}\q|$ (cf. Lemma \ref{trig}),
and the boundedness of $|\partial_{\varphi}\partial_{\theta}^i\q|$ and $\q$, we conclude that
\begin{align*}
& \bigg| \partial_{\varphi}\partial_{\theta}^N {\Big(\cosh\frac{t}2-1+\q\Big)^{-\alpha-\beta-2}}\bigg|\\
& \lesssim \sum_{k_1+2k_2+\ldots+Nk_N = N} \Bigg[\frac{\sqrt{\q}^{k_1+k_3+\ldots+k_{\tilde{N}}}}
	{(\cosh\frac{t}2-1+\q)^{\alpha+\beta+5\slash 2+k_1+\ldots+k_N}}
	+ \frac{\sqrt{\q}^{k_1+k_3+\ldots+k_{\tilde{N}}-1}}
	{(\cosh\frac{t}2-1+\q)^{\alpha+\beta+2+k_1+\ldots+k_N}}\Bigg]\\
& \lesssim \sum_{k_1+2k_2+\ldots+Nk_N = N} \frac{\sqrt{\q}^{k_1+k_3+\ldots+k_{\tilde{N}}-1}}
	{(\cosh\frac{t}2-1+\q)^{\alpha+\beta+5\slash 2+k_1+\ldots+k_N-1\slash 2}}.
\end{align*}
Now the last two estimates of the lemma follow as before.
\end{proof}

\begin{proof}[Proof of Theorem \ref{thm:stand}; the case of $R_N^{\ab}$]
We have
\begin{align*}
R_N^{\ab}(\theta,\varphi) & \lesssim \iint d\Pi_{\alpha}(u)d\Pi_{\beta}(v) \int_0^{\infty}
	\sinh\frac{t}2 \bigg| \partial_{\theta}^N {\Big(\cosh\frac{t}2-1+\q\Big)^{-\alpha-\beta-2}}\bigg|
	t^{N-1}\, dt \\
& \equiv \iint d\Pi_{\alpha}(u)d\Pi_{\beta}(v)\, (J_0+J_{\infty}),
\end{align*}
where $J_0$ and $J_{\infty}$ are the integrals in $t$ 
over $(0,1)$ and $(1,\infty)$, respectively. To bound these integrals, we apply Lemma \ref{derb} and get
\begin{align*}
J_0 & \lesssim \int_0^1 \frac{t^N\, dt}{(t^2+\q)^{\alpha+\beta+2+N\slash 2}} \lesssim
	\frac{1}{\q^{\alpha+\beta+3\slash 2}} \int_0^{\infty}\frac{s^N\, ds}{(1+s^2)^{\alpha+\beta+2+N\slash 2}}
	\lesssim \frac{1}{\q^{\alpha+\beta+3\slash 2}},\\
J_{\infty} & \lesssim \int_1^{\infty} \frac{e^{t\slash 2}t^{N-1}\, dt}
	{(e^{t\slash 2})^{\alpha+\beta+5\slash 2}} \lesssim 1 \lesssim
	\frac{1}{\q^{\alpha+\beta+3\slash 2}}.
\end{align*}
Thus
$$
|R_N^{\ab}(\theta,\varphi)| \lesssim \iint \frac{d\Pi_{\alpha}(u)d\Pi_{\beta}(v)}
	{{\q}^{\alpha+\beta+3\slash 2}},
$$
and the asserted growth condition \eqref{gr} (with $\mathbb{B}=\mathbb{C}$) follows from Lemma \ref{bridge}.

To prove the smoothness condition \eqref{sm}, we argue as above, this time using both estimates of 
Lemma \ref{derb} (the first one with $N$ replaced by $N+1$). We find that
$$
|\partial_{\theta}R_N^{\ab}(\theta,\varphi)|+|\partial_{\varphi}R_N^{\ab}(\theta,\varphi)| \lesssim 
	\iint \frac{d\Pi_{\alpha}(u)d\Pi_{\beta}(v)}{\q^{\alpha+\beta+2}},
$$
and this combined with Lemma \ref{bridge} leads to the desired conclusion.
\end{proof}

We finally deal with the $g$-functions based on the Jacobi-Poisson semigroup.
The kernel to be estimated is $\{\partial_t^M \partial_{\theta}^N H_t^{\ab}(\theta,\varphi)\}_{t>0}$
taking values in $\mathbb{B}=L^2(t^{2M+2N-1}dt)$. Here we consider $M,N=0,1,\ldots$ such that $M+N>0$,
so that the cases of the vertical and horizontal $g$-functions are included.
To proceed, we will need a generalization of Lemma \ref{derb}. Denote
$$
\Phi^{\ab}(t,\q) = \frac{\sinh\frac{t}2}{(\cosh\frac{t}2-1+\q)^{\alpha+\beta+2}}.
$$
Notice that in view of Proposition \ref{prop:rep_H}, this expression, integrated against
$d\Pi_{\alpha}(u)\, d\Pi_{\beta}(v)$, gives up to a constant factor the Jacobi-Poisson kernel.

\begin{lem} \label{general}
Let $\alpha,\beta \ge -1\slash 2$ and $M,N\ge 0$ be given. Then
\begin{equation} \label{e1}
\big|\partial_{\theta}^N \partial_t^{M} \Phi^{\ab}(t,\q)\big| \lesssim
\begin{cases}
(\cosh\frac{t}2-1+\q)^{-\alpha-\beta-3\slash 2-\frac{M+N}2}, & t \le 1 \\
(\cosh\frac{t}2-1+\q)^{-\alpha-\beta-3\slash 2}, & t > 1, \;\textrm{if}\;\; N\ge 1 \\
(\cosh\frac{t}2-1+\q)^{-\alpha-\beta-1}, & t > 1, \;\textrm{if}\;\; N=0 \\
(\cosh\frac{t}2-1+\q)^{-1}, & t > 1, \;\textrm{if}\;\; N = 0, \; M \ge 1, \; \alpha+\beta= -1
\end{cases}
\end{equation}
and
\begin{equation*} 
\big|\partial_{\varphi}\partial_{\theta}^N \partial_t^{M} \Phi^{\ab}(t,\q)\big| \lesssim
\begin{cases}
(\cosh\frac{t}2-1+\q)^{-\alpha-\beta- 2-\frac{M+N}2}, & t \le 1 \\
(\cosh\frac{t}2-1+\q)^{-\alpha-\beta-3\slash 2}, & t > 1
\end{cases}.
\end{equation*}
\end{lem}

\begin{proof}
We shall use Fa\`a di Bruno's formula \eqref{Faa} and the estimates from Lemma \ref{derb}.
Observe that $\Phi^{\ab}(t,\q)$ can be written, up to a constant factor, as
$$
\partial_t \Big(\cosh \frac{t}2-1+\q\Big)^{-\alpha-\beta-1}
$$
if $\alpha+\beta+1 > 0$, or as $\partial_t \ln (\cosh \frac{t}2-1+\q)$ if $\alpha+\beta+1=0$.
Applying \eqref{Faa} to $\partial_t^{M+1}(g \circ f)$ with $f(t)=\cosh\frac{t}2-1+\q$ and either
$g(x)=x^{-\alpha-\beta-1}$ or $g(x)=\ln x$, we see that $\partial^M_t \Phi^{\ab}(t,\q)$ 
is a linear combination of expressions of the form
\begin{equation} \label{form5}
\Big(\sinh\frac{t}2\Big)^{\sum_{\textrm{odd}\, i \le M+1}k_i}
\Big(\cosh\frac{t}2\Big)^{\sum_{\textrm{even}\, i \le M+1}k_i}
\Big(\cosh\frac{t}2-1+\q\Big)^{-\alpha-\beta-1-(k_1+\ldots + k_{M+1})},
\end{equation}
where $k_1,\ldots,k_{M+1}\ge 0$ satisfy the constraint $k_1+2k_2+\ldots + (M+1)k_{M+1}=M+1$.
From here the third estimate in \eqref{e1} readily follows.

To get the first bound in \eqref{e1}, we combine \eqref{form5} with the first bound in Lemma \ref{derb}
taken with $-\alpha-\beta-2$ replaced by $-\alpha-\beta-1-(k_1+\ldots + k_{M+1})$. 
The conclusion is that when $t\le 1$
\begin{align*}
& \big|\partial^N_{\theta}\partial^M_{t}\Phi^{\ab}(t,\q)\big|\\
& \lesssim \sum \Big(\sinh\frac{t}2\Big)^{\sum_{\textrm{odd}\, i \le M+1}k_i}
\Big(\cosh\frac{t}2\Big)^{\sum_{\textrm{even}\, i \le M+1}k_i}
\Big(\cosh\frac{t}2-1+\q\Big)^{-\alpha-\beta-1-(k_1+\ldots + k_{M+1})-N\slash 2},
\end{align*}
the sum running over $k_1,\ldots,k_{M+1}\ge 0$ such that $k_1+2k_2+\ldots + (M+1)k_{M+1}=M+1$. 
This leads to
$$
\big|\partial^N_{\theta}\partial^M_{t}\Phi^{\ab}(t,\q)\big| \lesssim
\sum \Big(\cosh\frac{t}2-1+\q\Big)^{-\alpha-\beta-1-N\slash 2-(k_1+\ldots + k_{M+1})
	+ \sum_{\textrm{odd}\, i \le M+1}k_i\slash 2},
$$
for $t \le 1$. Taking into account the boundedness of $\q$ and the fact that the constraint on
$k_1,\ldots,k_{M+1}$ forces 
$(k_1+\ldots + k_{M+1})-\sum_{\textrm{odd}\, i \le M+1}k_i\slash 2 \le (M+1)\slash 2$,
we get the first estimate in \eqref{e1}.

Justifying the second bound in \eqref{e1} goes along the same lines.
Combining \eqref{form5} with the second bound in Lemma \ref{derb}, we see that when $t>1$,
\begin{align*}
& \big|\partial^N_{\theta}\partial^M_{t}\Phi^{\ab}(t,\q)\big|\\
& \lesssim \sum \Big(\sinh\frac{t}2\Big)^{\sum_{\textrm{odd}\, i \le M+1}k_i}
\Big(\cosh\frac{t}2\Big)^{\sum_{\textrm{even}\, i \le M+1}k_i}
\Big(\cosh\frac{t}2-1+\q\Big)^{-\alpha-\beta-3\slash 2-(k_1+\ldots + k_{M+1})},
\end{align*}
the sum running over the same $k_1,\ldots,k_{M+1}$ as before. The conclusion follows.

The fourth bound in \eqref{e1} is slightly more subtle. When $\alpha=\beta=-1\slash 2$ there are
important cancellations between terms emerging in $\partial_t \Phi^{-1\slash 2,-1\slash 2}(t,\q)$.
A simple computation gives
$$
2\partial_t \Phi^{-1\slash 2,-1\slash 2}(t,\q) = \frac{1+(\q-1)\cosh\frac{t}2}{(\cosh\frac{t}2-1+\q)^2}.
$$
To analyze the $(M-1)$th derivative in $t$ of this expression, we view it as a product of the functions
$h_1(t)=1+(\q-1)\cosh\frac{t}2$ and $h_2(t)=(\cosh\frac{t}2-1+\q)^{-2}$ and then apply Leibniz'
rule to $h_1 h_2$ and Fa\`a di Bruno's formula to $h_2$. This shows that
$\partial_t^{M}\Phi^{-1\slash 2,-1\slash 2}(t,\q)$ is a linear combination of expressions
$\partial_t^k h_1(t) \partial_t^{M-1-k}h_2(t)$, $0 \le k \le M-1$, where
$$
\partial_t^k h_1(t) \simeq
\begin{cases}
1+(\q-1)\cosh\frac{t}2, & \textrm{if}\; k=0\\
(\q-1)\cosh\frac{t}2, & \textrm{if}\; k>0 \; \textrm{and} \; k \;\textrm{is even}\\
(\q-1)\sinh\frac{t}2, & \textrm{if}\;  k \;\textrm{is odd},
\end{cases}
$$
and $\partial_t^{M-1-k}h_2(t)$ is a linear combination of expressions
$$
\Big(\sinh\frac{t}2\Big)^{\sum_{\textrm{odd}\, i \le M-1-k}\ell_i}
\Big(\cosh\frac{t}2\Big)^{\sum_{\textrm{even}\, i \le M-1-k}\ell_i}
\Big(\cosh\frac{t}2-1+\q\Big)^{-2-(\ell_1+\ldots + \ell_{M-1-k})},
$$
with $\ell_1,\ldots,\ell_{M-1-k}\ge 0$, $\ell_1+2\ell_2+\ldots+(M-1-k)\ell_{M-1-k}=M-1-k$.
Now the desired conclusion follows from the boundedness of $\q$.

The remaining two bounds of the lemma are proved by combining \eqref{form5} with the last two estimates
of Lemma \ref{derb}. All the relevant arguments were already presented above.
Notice that since the derivative $\partial_{\varphi}$ is always present, the singular
cases connected with absence of the horizontal component ($N=0$) do not occur here.
\end{proof}

\begin{proof}[Proof of Theorem \ref{thm:stand}; the cases of $g_V^{\ab}$, $g_H^{\ab}$ and $g_{M,N}^{\ab}$]
By \eqref{PJker} and Minkowski's integral inequality
\begin{align*}
& \big\| \partial_{\theta}^N\partial_{t}^M H_t^{\ab}(\theta,\varphi)\big\|_{L^2(t^{2M+2N-1}dt)}\\
&	\lesssim \iint d\Pi_{\alpha}(u)d\Pi_{\beta}(v) \bigg( \int_0^{\infty}\big( 
	\partial_{\theta}^N\partial_t^M \Phi^{\ab}(t,\q)\big)^2 t^{2M+2N-1}dt\bigg)^{1\slash 2}.
\end{align*}
We split the inner integral in $t$ according to the intervals $(0,1)$ and $(1,\infty)$ and denote
the resulting integrals by $J_0$ and $J_{\infty}$, respectively. Then by Lemma \ref{general}
and the change of variable $t=\sqrt{\q}s$,
$$
J_0 \lesssim \int_0^1 \frac{t^{2M+2N-1}dt}{(t^2+\q)^{2\alpha+2\beta+3+M+N}}
	\le \frac{1}{\q^{2\alpha+2\beta+3}} \int_0^{\infty} 
		\frac{s^{2M+2N-1}ds}{(1+s^2)^{2\alpha+2\beta+3+M+N}} \simeq \frac{1}{\q^{2\alpha+2\beta+3}}
$$
and, taking in addition the boundedness of $\q$ into account,
$$
J_{\infty} \lesssim \int_1^{\infty} \frac{t^{2M+2N-1}dt}{e^{\xi t}} \lesssim 
	\frac{1}{\q^{2\alpha+2\beta+3}}
$$
for some constant $\xi=\xi(\alpha,\beta)$, and $\xi>0$ in all cases. Therefore
$$
\big\| \partial_{\theta}^N\partial_{t}^M H_t^{\ab}(\theta,\varphi)\big\|_{L^2(t^{2M+2N-1}dt)} \lesssim
	\iint \frac{d\Pi_{\alpha}(u)d\Pi_{\beta}(v)}{\q^{\alpha+\beta+3\slash 2}},
$$
and the growth condition \eqref{gr} with $\mathbb{B}=L^2(t^{2M+2N-1}dt)$ follows from Lemma \ref{bridge}.

To prove the smoothness conditions \eqref{sm1} and \eqref{sm2}, we first use Lemma \ref{general} to
bound the relevant derivatives, getting
\begin{align*}
& \big|\partial_{\theta}\partial_{\theta}^N\partial_t^M H_t^{\ab}(\theta,\varphi)\big|
+ \big|\partial_{\varphi}\partial_{\theta}^N\partial_t^M H_t^{\ab}(\theta,\varphi)\big| \\
& \qquad \lesssim 
\begin{cases}
\iint d\Pi_{\alpha}(u)d\Pi_{\beta}(v) \;(\cosh\frac{t}2-1+\q)^{-\alpha-\beta-2-\frac{M+N}2}, & t \le 1\\
\iint d\Pi_{\alpha}(u)d\Pi_{\beta}(v) \; (\cosh\frac{t}2-1+\q)^{-\alpha-\beta-3\slash 2}, & t > 1
\end{cases}.
\end{align*}
Then the Mean Value Theorem, Lemma \ref{lem:comp} and the assumptions $|\theta-\varphi|>2|\theta-\theta'|$,
$|\theta-\varphi|>2|\varphi-\varphi'|$ (considered separately for \eqref{sm1} and \eqref{sm2},
respectively) lead to the estimates
\begin{align*}
& \big| \partial_{\theta}^N\partial_t^M H_t^{\ab}(\theta,\varphi)-
	\partial_{\theta}^N\partial_t^M H_t^{\ab}(\theta',\varphi)\big| \\
& \qquad \lesssim
\begin{cases}
|\theta-\theta'|
\iint d\Pi_{\alpha}(u)d\Pi_{\beta}(v) \;(\cosh\frac{t}2-1+\q)^{-\alpha-\beta-2-\frac{M+N}2}, & t \le 1\\
|\theta-\theta'|
\iint d\Pi_{\alpha}(u)d\Pi_{\beta}(v) \; (\cosh\frac{t}2-1+\q)^{-\alpha-\beta-3\slash 2}, & t > 1
\end{cases},
\end{align*}
and
\begin{align*}
& \big| \partial_{\theta}^N\partial_t^M H_t^{\ab}(\theta,\varphi)-
	\partial_{\theta}^N\partial_t^M H_t^{\ab}(\theta,\varphi')\big| \\
& \qquad \lesssim
\begin{cases}
|\varphi-\varphi'|
\iint d\Pi_{\alpha}(u)d\Pi_{\beta}(v) \;(\cosh\frac{t}2-1+\q)^{-\alpha-\beta-2-\frac{M+N}2}, & t \le 1\\
|\varphi-\varphi'|
\iint d\Pi_{\alpha}(u)d\Pi_{\beta}(v) \; (\cosh\frac{t}2-1+\q)^{-\alpha-\beta-3\slash 2}, & t > 1
\end{cases}.
\end{align*}
Proceeding as in the first part of the proof, we get
\begin{align*}
\big\|\partial_{\theta}^N\partial_t^M H_t^{\ab}(\theta,\varphi)-
	\partial_{\theta}^N\partial_t^M H_t^{\ab}(\theta',\varphi)\big\|_{L^2(t^{2M+2N-1}dt)}
& \lesssim |\theta-\theta'|\iint \frac{d\Pi_{\alpha}(u)d\Pi_{\beta}(v)}{\q^{\alpha+\beta+2}}, \\
\big\|\partial_{\theta}^N\partial_t^M H_t^{\ab}(\theta,\varphi)-
	\partial_{\theta}^N\partial_t^M H_t^{\ab}(\theta,\varphi')\big\|_{L^2(t^{2M+2N-1}dt)}
& \lesssim |\varphi-\varphi'|\iint \frac{d\Pi_{\alpha}(u)d\Pi_{\beta}(v)}{\q^{\alpha+\beta+2}}.
\end{align*}
An application of Lemma \ref{bridge} now finishes the proof.
\end{proof}


\end{document}